\documentclass[10pt]{amsart}
\usepackage{amssymb,amsmath,amsfonts,amscd}

\usepackage[usenames, dvipsnames]{color}
\definecolor{orange}{rgb}{1,0.5,0}

\theoremstyle{plain}
\newtheorem{theorem}{Theorem}[section]
\newtheorem{proposition}[theorem]{Proposition}
\newtheorem{definition}[theorem]{Definition}

\newtheorem{corollary}[theorem]{Corollary}
\newtheorem{remark}[theorem]{Remark}
\newtheorem{lemma}[theorem]{Lemma}

\definecolor{lavanda}{rgb}{0.66, 0.73, 1.0}

\def\eps{\epsilon}

\usepackage{pdfsync}

\newcommand{\rn}[1]{{\mathbb R}^{#1}}
\newcommand{\R}{\mathbb R}

\newcommand{\G}{\mathbb G}

\newcommand{\supp}{\mathrm{supp}\;}

\newcommand{\he}[1]{{\mathbb H}^{#1}}

\newcommand{\scal}[2]{\langle {#1} , {#2}\rangle}

\newcommand{\Scal}[2]{\langle {#1} \vert {#2}\rangle}
\newcommand{\scalp}[3]{\langle {#1} , {#2}\rangle_{#3}}

\newcommand{\ccheck}{{\vphantom i}^{\mathrm v}\!\,}
\newcommand{\mc}{\mathcal }

\newcommand{\mfrak}{\mathfrak}

\newcommand{\covH}[1]{{\bigwedge\nolimits^{#1}{\mfrak h}}}

\newcommand{\N}{\mathbb N}

\usepackage[usenames]{color}

\begin{document}


\title[
] 
{$L^1$-Poincar\'e inequalities for differential forms on Euclidean spaces and Heisenberg groups}

\author[Annalisa Baldi, Bruno Franchi, Pierre Pansu]{
Annalisa Baldi\\
Bruno Franchi\\ Pierre Pansu
}
\begin{abstract}
In this paper, we prove interior Poincar\'e and Sobolev inequalities in Euclidean spaces and in Heisenberg groups, in the limiting case where the exterior (resp. Rumin) differential of a differential form is measured in $L^1$ norm. Unlike for $L^p$, $p>1$, the estimates are doomed to fail in top degree. The singular integral estimates are replaced with inequalities which go back to Bourgain-Brezis in Euclidean spaces, and to Chanillo-van Schaftingen in Heisenberg groups. 
\end{abstract}
 


\maketitle

%

\section{Introduction}

\subsection{$L^1$-Sobolev and Poincar\'e inequalities}

The well known Sobolev inequalities on $\R^n$ states that for every $1\leq p<n$, there exists a constant $C(n,p)$ such that all smooth compactly supported functions $u$ on $\R^n$ satisfy
$$
\|u\|_q\leq C(n,p)\,\|\nabla u\|_p \quad \text{provided}\quad \frac{1}{p}-\frac{1}{q}=\frac{1}{n} \qquad(p\!-\!Sobolev).
$$
The most important of these inequalities is ($1\!-\!Sobolev$). Indeed, ($1\!-\!Sobolev$) implies all inequalities ($p\!-\!Sobolev$), $p<n$. Furthermore, ($1\!-\!Sobolev$) is equivalent to the isoperimetric inequality for smooth bounded domains $A$ of $\R^n$ (Federer-Fleming's theorem, \cite{federer_fleming}),
$$
\textrm{volume}(A)^{(n-1)/n}\leq C(n,1)\,\textrm{area}(\partial A),
$$
(with the same constant). Similarly, for noncompactly supported functions, a Poincar\'e inequality holds for $1\leq p<n$: there exists a constant $c_u$ such that 
$$
\|u-c_u\|_q\leq C(n,p)\,\|\nabla u\|_p \quad \text{provided}\quad \frac{1}{p}-\frac{1}{q}=\frac{1}{n}\qquad(p\!-\!Poincar\acute{e}).
$$

We investigate generalizations of these inequalities to differential forms. More precisely, we ask whether, given a closed differential $h$-form $\omega$ in $L^p(\R^n)$,  there exists an $(h-1)$-form $\phi$ in $L^q(\R^n)$ with $\frac{1}{p}-\frac{1}{q}=\frac{1}{n}$ such that $d\phi=\omega$ and 
$$
\|\phi\|_q\leq C(n,p,h)\,\|\omega\|_p.
$$
If $p>1$, the easy proof consists in putting $\phi=d^*\Delta^{-1}\omega$. Here, $\Delta^{-1}$ denotes the inverse of the Hodge Laplacian $\Delta=d^*d+dd^*$
and $d^*$ is the formal $L^2$-adjoint of $d$. The operator $d^*\Delta^{-1}$ is given by convolution with a homogeneous kernel of type $1$ in the terminology of \cite{folland} and \cite{folland_stein}, hence it is bounded from $L^p$ to $L^q$ if $p>1$. 
Unfortunately, this argument does not suffice for $p=1$ since, by \cite{folland_stein}, Theorem 6.10,  $d^*\Delta^{-1}$ maps $L^1$ only into the weak Marcinkiewicz space $L^{n/(n-1), \infty}$. Upgrading from $L^{n/(n-1), \infty}$ to $L^{n/(n-1)}$ is possible for functions. Indeed, for characteristic functions of sets, the $L^{n/(n-1), \infty}$ and $L^{n/(n-1)}$ norms coincide, and every function is the sum of characteristic functions of its superlevel sets (see \cite{LN}, \cite{FGaW}, \cite{FLW_grenoble}). 

This trick does not seem to generalize to differential forms.

Note that locally, $d^*\Delta^{-1}$ maps $L^1$ to $L^q$ for all $q<n/(n-1)$, but this does not lead to a scale invariant inequality. 

\subsection{Analysis of $L^1$-differential forms}

In fact, ($1\!-\!Poincar\acute{e}$) fails in degree $n$. There is an obvious obstruction:  $n$-forms belonging to $L^1$ and with nonvanishing integral cannot be differentials of $L^{n/(n-1)}$ forms, see \cite{Tripaldi}. But even if integral vanishes, a primitive $\phi$ such that $\|\phi\|_{q}\leq C\,\|\omega\|_1$ need not exist, with $1-\frac{1}{q}=\frac{1}{n}$. Indeed, if so, then, for every smooth function $u$ on $\R^n$, one could write, for every $n$-form $\omega\in L^1$ with vanishing integral,
$$
|\int u\omega|=|\int u\,d\phi|=|du\wedge\phi|\leq \|du\|_n\|\phi\|_{q}\leq C\,\|du\|_n\|\omega\|_{1},
$$
which would imply (by Hahn-Banach theorem) the existence of a constant $c_u$ such that $\|u-c_u\|_{\infty}\leq C\,\|du\|_n$. Such a ($n\!-\!Sobolev$) inequality does not hold, since $\R^n$ is $n$-parabolic, i.e. for every compact subset $K$ and every $\eps>0$, there exists a smooth compactly supported function $\chi$ on $\R^n$ such that $\chi\geq 1$ on $K$ and $\int_{\R^n}|d\chi|^n <\eps$,
(see \cite{evans_gariepy} Section 4.7).

Surprisingly, Poincar\'e and Sobolev inequalities persist sometimes for $p=1$. The first result appeared in \cite{Bourgain-Brezis-JEMS}, whose Theorem 2 states that, if $\vec{f}$ is a divergence free vectorfield in $L^1(\R^n)$, then the solution of $\Delta\vec{u}=\vec{f}$ satisfies $\nabla\vec{u}\in L^{n/(n-1)}$. In differential form notation, this means that $\nabla\Delta^{-1}$ restricted to closed $(n-1)$-forms is bounded from $L^1$ to $L^{n/(n-1)}$. \emph{A fortiori}, so is $d^*\Delta^{-1}$, this proves ($1\!-\!Poincar\acute{e}$) in degree $n-1$. 

\subsection{Results}

In this paper, we prove ($1\!-\!Poincar\acute{e}$) for $h$-forms of degree $h<n$
in de Rham's complex $(\Omega^\bullet,d)$.  We rely on Lanzani-Stein's observation (see \cite{LS}) that the duality estimate (emphasized by van Schaftingen \cite{vS2014}) underlying Bourgain-Brezis' result descends from $(n-1)$-forms to forms of lower degree, and the resulting Gagliardo-Nirenberg inequalities. 

Remarkably, this approach generalizes to the non-commutative Heisenberg groups $\he n$ 
equipped with Rumin's complex $(E_0^\bullet,d_c)$. Indeed, when passing to Heisenberg groups, we can
use Lanzani-Stein's type arguments proved in \cite{BF7}, \cite{BFP2}. Precise definitions of Heisenberg groups and related properties as
well as of Rumin's complex, can be found in Section \ref{heisenberg}. 

In the Euclidean setting, the integral obstruction generalizes to forms in every degree: if a closed 
$L^1$-form $\omega$ is the differential of a form in $L^{n/(n-1)}(\R^n)$, then for every constant 
coefficient form $\beta$ of complementary degree, $\int\omega\wedge\beta=0$. Therefore we 
introduce the subspace $L_0^1$ of $L^1$-differential forms satisfying these conditions (we call 
them forms \emph{with vanishing averages}).
In Heisenberg groups, constant coefficient forms must be replaced with left-invariant Rumin forms.

We can state our main results. We stress that, in (1) below we are dealing with usual de Rham forms, whereas in
 (2) we are dealing with Rumin's complex.
\begin{theorem}[Global Poincar\'e and Sobolev inequalities]\label{poincareglobal}

We have:

\begin{enumerate}
 \item Euclidean case $\R^n$. Let $h=1,\ldots,n-1$ and set $q=n/(n-1)$. For every closed $h$-form $\alpha\in L_0^1(\R^n)$, there exists an $(h-1)$-form $\phi\in L^{q}(\R^n)$, such that 
$$
d\phi=\alpha\qquad\mbox{and}\qquad \|\phi\|_{q}\leq C\,\|\alpha\|_{1}.
$$
\item Heisenberg case $\he n \equiv \R^{2n+1}$. Let $h=1,\ldots,2n$ and set $q=(2n+2)/(2n+1)$ if $h\not=n+1$ and $q=(2n+2)/(2n)$ if $h=n+1$. For every $d_c$-closed $h$-form $\alpha\in L_0^1(\he n)$, there exists an $(h-1)$-form $\phi\in L^q(\he n)$, such that 
$$
d_c\phi=\alpha \qquad\mbox{and}\qquad \|\phi\|_{q}\leq C\,\|\alpha\|_{1}.
$$
Furthermore, in both cases, if $\alpha$ is compactly supported, so is $\phi$.
\end{enumerate}

\end{theorem}

We also prove local versions of these inequalities, of the following types (see Corollary \ref{poincare}). 

\begin{theorem}\label{poincareintro}
\begin{enumerate}
  \item Euclidean case.
For $h=1,\ldots,n-1$, let $q=n/(n-1)$. For every $\lambda>1$, there exists $C$ with the following property. Let $B(R)$ be a ball of radius $R$ in $\R^n$. 
  \begin{enumerate}
  \item Interior Poincar\'e inequality:
for every closed $h$-form $\alpha\in L^1(B(\lambda R))$, there exists an $(h-1)$-form $\phi\in L^q(B(R))$, such that 
$$
d\phi=\alpha_{|B(R)} \qquad\mbox{and}\qquad \|\phi\|_{L^q(B(R))}\leq C\,\|\alpha\|_{L^1(B(\lambda R))}.
$$
 
\item Sobolev inequality: for every closed $h$-form $\alpha\in L^1$ with support in $B(R)$, there exists an $(h-1)$-form $\phi\in L^q$, with support in $B(\lambda R)$, such that 
$$
d\phi=\alpha \qquad\mbox{and}\qquad \|\phi\|_{L^q(B(\lambda R))}\leq C\,\|\alpha\|_{L^1(B(R))}.
$$
\end{enumerate}
  
  \item Heisenberg case:
for $h=1,\ldots,2n$, let $q=(2n+2)/(2n+1)$ if $h\not=n+1$ and $q=(2n+2)/(2n)$ if $h=n+1$. There exist $\lambda>1$ and $C$ with the following property. Let $B(R)$ be a ball of radius $R$ in $\he n$. 
\begin{enumerate}
  \item Interior Poincar\'e inequality.
For every $d_c$-closed Rumin $h$-form $\alpha \in L^1(B(\lambda R))$, there exists an $(h-1)$-form $\phi\in L^q(B(R))$, such that 
$$
d_c\phi=\alpha_{|B(R)} \qquad\mbox{and}\qquad \|\phi\|_{L^q(B(R))}\leq C\,\|\alpha\|_{L^1(B(\lambda R))}.
$$
 
\item Sobolev inequality: for every $d_c$-closed Rumin $h$-form $\alpha\in L^1$ with support in $B(R)$, there exists an $(h-1)$-form $\phi\in L^q$, with support in $B(\lambda R)$, such that 
$$
d_c\phi=\alpha \qquad\mbox{and}\qquad \|\phi\|_{L^q(B(\lambda R))}\leq C\,\|\alpha\|_{L^1(B(R))}.
$$
\end{enumerate}
\end{enumerate}
\end{theorem}

Finally, we construct smoothing homotopies on Riemannian or contact subRiemannian manifolds of bounded geometry
(see \cite{dR}, Proposition 1, p. 77). Roughly speaking, a Riemannian manifold has $C^k$-bounded geometry if it admits an atlas of charts defined on the unit Euclidean ball, with uniformly bounded Lipschitz constant, and such that changes of charts have uniformly bounded derivatives up to order $k$. In the contact subRiemannian case, the models are unit Heisenberg balls, the charts are assumed to be contactomorphisms and only horizontal derivatives play a role. Details appear in Definition \ref{bded geometry bis}.

\begin{theorem}\label{bddgeometryintro}

\begin{enumerate}
\item Riemannian case: let $M$ be a Riemannian manifold of dimension $2n+1$ and bounded $C^k$-geometry,
where $k$ is an integer, $k\ge 2$. For $h=1,\ldots,n-1$, let $q=n/(n-1)$. Let $1\leq q'\leq q$. There exist operators $S$ and $T$ on $h$-forms on $M$ such that $S$ is bounded from 
$L^{1}$ to $W^{k-1,q'}$, $T$ is bounded from
$L^{1}\cap d^{-1}L^1$ to $L^{q'}$, and the homotopy identity $1=S+dT +Td$ holds on $L^1\cap d^{-1}L^1$. 

  \item SubRiemannian contact case: let $M$ be a subRiemannian contact manifold of dimension $2n+1$ and bounded $C^k$-geometry, where $k$ is an integer, $k\ge 3$. For $h=1,\ldots,2n$, let $q=(2n+2)/(2n+1)$ if $h\not=n+1$ and $q=(2n+2)/(2n)$ if $h=n+1$. Let $1\leq q'\leq q$. There exist operators $S$ and $T$ on $h$-forms on $M$ such that $S$ is bounded from 
$L^{1}$ to $W^{k-1,q'}$, $T$ is bounded from
$L^{1}\cap d^{-1}L^1$ to $L^{q'}$, and the homotopy identity $1=S+d_c T +Td_c$ holds on $L^1\cap d^{-1}_c L^1$. Furthermore, in degree $h=n+1$, $T$ is bounded from $W^{j-1,1}$ to $W^{j,1}$ for all $1\leq j\leq k-1$.
\end{enumerate}
\end{theorem}

Such local Poincar\'e inequalities and smoothing homotopies are the necessary ingredients in order to prove that Rumin's complex can be used to compute the $\ell^{q,1}$-cohomology of a subRiemannian contact manifold, see \cite{PPcup}. Therefore Theorem \ref{poincareglobal} has significance in geometric group theory, see Corollary \ref{GGT}.

This paper is organised as follows: in Section \ref{scheme} we provide a sketch of the proof
of Theorems \ref{poincareglobal} and \ref{poincare}. Section \ref{kernels} deals with continuity properties of homogeneous kernels in Carnot groups
and with function spaces. Most of the results are more or less known, except, as long as we know, for  Theorem \ref{anuli}.  Preliminary results on 
Heisenberg groups, Rumin's complex and Laplacians are gathered in Section \ref{heisenberg}. The proof of Theorem \ref{poincareglobal} is contained
in Section \ref{GNI} and relies on Gagliardo-Nirenberg type inequalities proved therein, and interior inequalities stated in Theorem \ref{poincare} are
proved in Section \ref{InIn} via suitable smoothing homotopy formulas. Finally, Sections \ref{bounded geometry} and \ref{SH} deal with
Riemannian and contact manifolds with bounded geometry.

\section{Scheme of proof}\label{scheme}

{  In this Section we sketch the proof of Theorems with more details in the Euclidean case,
whereas  the body of this paper will contain only the proofs
for differential forms in Heisenberg groups which require several further arguments.}

\subsection{Euclidean case}
Let $q=n/(n-1)$. According to Lanzani-Stein, in degrees $<n$, for smooth compactly supported forms $u$,
\begin{eqnarray}\label{GN}
\|u\|_q \leq C\,(\|du\|_1+\|d^*u\|_N),
\end{eqnarray}
where $\|\cdot\|_N$ denotes either $L^1$-norm (in degrees $\not=1$) or the norm of the real Hardy space $\mathcal{H}^1$ (in degree $1$).
Since the inverse of the Laplacian, $\Delta^{-1}$, commutes with $d$, the operator $K=d^*\Delta^{-1}$ satisfies $d K+Kd=1$ on smooth compactly supported forms. Given a closed form $\alpha\in L^1(\R^n)$, $u=K\alpha$ is not compactly supported, so cannot be directly plugged in (\ref{GN}). Therefore we use a smooth cut-off function $\chi$ and put
$$
\phi=d^*(\chi\Delta^{-1}\alpha).
$$
Then $\phi$ has compact support, $d^*\phi=0$ and 
$$
d\phi=[dd^*,\chi]\Delta^{-1}\alpha+\chi dd^*\Delta^{-1}\alpha=[dd^*,\chi]\Delta^{-1}\alpha+\chi \alpha.
$$
The point is to estimate the garbage term $\|[dd^*,\chi]\Delta^{-1}\alpha\|_1$. Notice that $[dd^*,\chi]$ is a first order differential operator, of the form $[dd^*,\chi]=P_0+P_1$ where $P_0$ has order $0$ and depends on second derivatives $\nabla^2\chi$ and $P_1$ has order $1$ and depends on first derivatives $\nabla\chi$ only. Both $P_0 \Delta^{-1}$ and $P_1 \Delta^{-1}$ have homogeneous kernels. 

Here comes our key trick. If $P$ is the operator of convolution with a kernel of type $\mu>0$, and $\alpha\in L^1$, then the $L^1$ norm of $P\alpha$ on shells $B(0,2R)\setminus B(0,R)$ is $O(R^\mu)$. If furthermore $\alpha\in L_0^1$, this can be improved to $o(R^\mu)$. 

Pick $\chi=\chi_R$ such that $d\chi_R$ is supported in the shell $B(0,2R)\setminus B(0,R)$, $|\nabla\chi_R |\leq \frac{1}{R}$ and $|\nabla^2\chi_R|\leq \frac{1}{R^2}$. Then $\|P_0 \Delta^{-1}\alpha\|_1$ and $\|P_1 \Delta^{-1}\alpha\|_1$ tend to $0$ as $R\to\infty$. Then $\|\phi\|_q$ stays uniformly bounded, yielding eventually that $d^*\Delta^{-1}\alpha\in L^q$, thanks to Fatou's theorem.

The local Poincar\'e inequality is based on Iwaniec-Lutoborsky's homotopy, \cite{IL}. This homotopy is defined by a kernel $k$ which belongs to $L^q$ in a neighborhood of the origin, for every $q<n/(n-1)$, but not for $q=n/(n-1)$. Fortunately, Young's inequality suffices to prove that a truncation of $k$ maps $L^1$ to $L^1$. This provides an $L^1$ local primitive for a closed form, up to a smoothed closed form, which belongs to $W^{1,1}$. The $L^1$ primitive is upgraded to $L^{n/(n-1)}$ using a cut-off and Theorem \ref{poincareglobal}. To the smoothed form, one can again apply Iwaniec-Lutoborsky's homotopy, which yields a form in $W^{1,1}$. The Sobolev embedding $W^{1,1}\subset L^{n/(n-1)}$ concludes the argument.

For further details in the Euclidean case, we refer to \cite{BFP4}.
\subsection{Heisenberg case}

We use Rumin's Laplacian $\Delta_{\mathbb H}$ on Rumin forms. It does not quite commute with Rumin's differential $d_c$ in degrees $n-1$ and $n+2$ but this turns out to be harmless. Write $K=d_c^*\Delta_{\mathbb H}^{-1}$ (with a modification in degrees $n$ and $n+1$), in order that $d_c K+Kd_c=1$ on smooth compactly supported forms. In spite of the complicated form of Leibniz' formula for $d_c$, the basic features of commutators $[d_c d_c^*,\chi]\Delta_{\mathbb H}^{-1}$ from the Euclidean case persist.

The local Poincar\'e inequality requires special care in the Heisenberg case, 
since no analogue of Iwaniec-Lutoborsky's homotopy exists. The kernel of 
$K=d_c^*\Delta_{\mathbb H}^{-1}$ is a valuable replacement. This provides again
a $L^1$ local primitive for a $d_c$-closed form, up to a smoothed $d_c$-closed form, which belongs to $W^{3,1}$. The $L^1$ primitive is upgraded to $L^q$ using a cut-off and Theorem \ref{poincareglobal} in the same manner. To the smoothed form, one can apply Rumin's homotopy, yielding a $W^{2,1}$ $d_c$-closed form, and then Iwaniec-Lutoborsky's Euclidean homotopy. The resulting form belongs to $L^q$, with $q=(2n+2)/(2n+1)$ if $h\not=n+1$ and $q=(2n+2)/(2n)$ if $h=n+1$, again by Sobolev embedding.

\subsection{Gaffney type inequality in Euclidean spaces}

If $p>1$, an alternative route to Poincar\'e's inequality could be to first establish a Gaffney type inequality: for every differential form $\phi$ such that $d\phi$ and $\delta\phi\in L^p$, 
$$
\|\nabla\phi\|_p\leq C\,(\|d\phi\|_p +\|\delta\phi\|_p).
$$
Combined with ($p\!-\!Poincar\acute{e}$) inequality for functions, $\|\phi-c_\phi\|_{np/(n-p)}\leq C\,\|\nabla\phi\|_p$, this implies ($p\!-\!Poincar\acute{e}$) for forms. Unfortunately, if $p=1$, Gaffney's inequality trivially holds for forms of degree $0$, but fails in every degree $\ge 1$. This follows from Ornstein's non-inequality, \cite{Ornstein}. Indeed, in degrees $\ge 1$, $\nabla\phi=\delta\phi+ d\phi+R\phi$, where all three components constitute a linearly independent collection of linear first order constant coefficient differential operators on $\R^n$. Therefore no universal inequality 
$$
\|R\phi\|_1\leq C\,(\|d\phi\|_1 +\|\delta\phi\|_1)
$$
can hold, even for forms with compact support in a fixed ball.

However, the following statement is still open for $h\not=1,n$ : in $\R^n$, for every closed differential $h$-form $\omega$ in $L^1$, does there exist an $(h-1)$-form $\phi$ such that $d\phi=\omega$ and
$$
\|\nabla\phi\|_1\leq C\,\|\omega\|_1\quad ?
$$
This is true if $L^1$ is replaced with Hardy space $\mathcal{H}^1$.

\section{Kernels}\label{kernels}

In Theorem \ref{poincareglobal}, the primitive $\phi$ of a closed form $\omega$ is provided by an operator defined by convolution with a homogeneous (matrix valued) function. We collect in this section the classical properties of such operators, especially their boundedness in function spaces in the Lebesgue and Sobolev scales. A special care will be taken of boundedness on $L^1$, a fact which is not standard.

This section applies to the wider class of Carnot groups, which contains both abelian and Heisenberg groups.

\subsection{Convolutions on Carnot groups}\label{7.1}

A {\it Carnot group $\G$ of
step $\kappa$}  is a connected, simply connected
Lie group whose Lie algebra
${\mathfrak{g}}$ admits a {\it step $\kappa$ stratification}, i.e.
there exist linear subspaces $V_1,...,V_\kappa$ such that
\begin{equation}\label{stratificazione}
{\mathfrak{g}}=V_1\oplus...\oplus V_\kappa,\quad [V_1,V_i]=V_{i+1},\quad
V_\kappa\neq\{0\},\quad V_i=\{0\}{\textrm{ if }} i>\kappa,
\end{equation}
where $[V_1,V_i]$ is the subspace of ${\mathfrak{g}}$ generated by
the commutators $[X,Y]$ with $X\in V_1$ and $Y\in V_i$. 
The exponential map is a one to one map from $\mathfrak
g$ onto $\G$. Using {\it exponential
coordinates}, we identify a point $p\in\G$ with the $N$-tuple $(p_1,\dots,p_N)\in
\R^N$ and we identify $\G$ with $(\R^N,\cdot)$ where the explicit
expression of the group operation $\cdot$ is determined by the
Campbell-Hausdorff formula (see, e.g., \cite{folland_stein}). In exponential coordinates
the unit element $e$ of $\G$ is $e=(0,\dots,0)$.

The first layer $V_1$ will be called {\sl horizontal layer}; a left-invariant
vector field in $V_1$, identified with a differential operator, will be
called an {\sl horizontal deerivative}.

From now on, we shall denote by $\{W_1,\dots,W_m\}$ a basis of $V_1$.

The $N$-dimensional Lebesgue measure $\mathcal L^n$, is the Haar
measure of the group $\G$.
For any $\lambda >0$, the
{\it dilation} $\delta_\lambda:\G\to\G$, is defined as
\begin{equation}\label{dilatazioni}
\delta_\lambda(x_1,...,x_N)=
(\lambda^{d_1}x_1,...,\lambda^{d_N}x_N),
\end{equation} where $d_i\in\N$ is called the {\it homogeneity of
the variable} $x_i$ in
$\G$ (see \cite{folland_stein} Chapter 1).
We denote by $Q$ the {\sl homogeneous dimension
of $\G$} defined by
\begin{equation}\label{dimensione}
Q:=\sum_{i=1}^\kappa i\,\text{dim}\,V_i.
\end{equation}

Through this paper we shall assume that $Q\ge 3$.

In this paper we denote by $ |\cdot|$ a homogeneous norm, smooth outside the origin, that
induces a genuine distance on $\G$ as in \cite{Stein}, p. 638. In the special case of $\G=\he n$, the
$n$-th Heisenberg group, this homogeneous norm is the Kor\'anyi norm $\rho$ (see \eqref{gauge}). Later on, we shall use the following
gauge distance:
$$
d(x,y)=|y^{-1}x|,
$$
and we denote by $B(x,R)$   the $d$-ball of radius R centred at $x$.

    Following e.g. \cite{folland_stein}, we can define a group
convolution in $\G$: if, for instance, $f\in\mc D(\G)$ and
$g\in L^1_{\mathrm{loc}}(\G)$, we set
\begin{equation}\label{group convolution}
f\ast g(p):=\int f(q)g(q^{-1}\cdot p)\,dq\quad\mbox{for $q\in \G$}.
\end{equation}
We remind that, if (say) $g$ is a smooth function and $P$
is a left invariant differential operator, then
$$
P(f\ast g)= f\ast Pg.
$$
  If $f$ is a real function defined in $\G$, we denote
    by $\ccheck f$ the function defined by $\ccheck f(p):=
    f(p^{-1})$, and, if $T\in\mc D'(\G)$, then $\ccheck T$
    is the distribution defined by $\Scal{\ccheck T}{\phi}
    :=\Scal{T}{\ccheck\phi}$ for any test function $\phi$.

We remind also that the convolution is again well defined
when $f,g\in\mc D'(\G)$, provided at least one of them
has compact support. In this case the following identities
hold
\begin{equation}\label{convolutions var}
\Scal{f\ast g}{\phi} = \Scal{g}{\ccheck f\ast\phi}
\quad
\mbox{and}
\quad
\Scal{f\ast g}{\phi} = \Scal{f}{\phi\ast \ccheck g}
\end{equation}
 for any test function $\phi$, where we use the notation $\Scal{\cdot}{\cdot}$ for the duality 
 between $\mc D' $ and $\mc D$.
 
{  As in \cite{folland_stein},
we also adopt the following multi-index notation for higher-order derivatives. If 
$
I =
(i_1,\dots,i_{2n+1})
$ 
is a multi--index, we set  
$W^I=W_1^{i_1}\cdots
W_{2n}^{i_{2n}}\;T^{i_{2n+1}}$. 
By the Poincar\'e--Birkhoff--Witt theorem, the differential operators $W^I$ form a basis for the algebra of left invariant
differential operators in $\G$. 
Furthermore, we set 
$$
|I|:=i_1+\cdots +i_{2n}+i_{2n+1}
$$
 the order of the differential operator
$W^I$, and   
$$
d(I):=i_1+\cdots +i_{2n}+2i_{2n+1}$$
 its degree of homogeneity
with respect to group dilations.

 Suppose now $f\in\mc E'(\G)$ and $g\in\mc D'(\G)$. Then,
 if $\psi\in\mathcal D(\G)$, we have
 $$\Scal{(W^If)\ast g}{\psi}=
 \Scal{W^If}{\psi\ast \ccheck g} =
 (-1)^{|I|}  \Scal{f}{\psi\ast (W^I \,\ccheck g)} =
 (-1)^{|I|} \Scal{f\ast \ccheck W^I\,\ccheck g}{\psi}.$$
 Thus
 \begin{equation}\label{convolution by parts}
 \begin{split}
\Scal{(W^If)\ast g}{\psi}&=
 \Scal{W^If}{\psi\ast \ccheck g} =
  (-1)^{|I|}  \Scal{f}{\psi\ast (W^I \,\ccheck g)} \\
&=
 (-1)^{|I|} \Scal{f\ast \ccheck W^I\,\ccheck g}{\psi}.
\end{split}
\end{equation}
}

\medskip

\subsection{Kernels, basic properties}\label{7.2}
Following \cite{folland}, we remind now the notion of {\it kernel of type $\mu$}
and some properties stated below in Proposition \ref{kernel}.

\begin{definition}\label{type} A kernel of type $\mu$ is a 
homogeneous distribution of degree $\mu-Q$
(with respect to group dilations),
that is smooth outside of the origin.

The convolution operator with a kernel of type $\mu$
is still called an operator of type $\mu$.
\end{definition}

\begin{proposition}\label{kernel}
Let $K\in\mc D'(\G)$ be a kernel of type $\mu$.
\begin{itemize}
\item[i)] $\ccheck K$ is again a kernel of type $\mu$;
\item[ii)] $WK$ and $KW $ are associated with  kernels of type $\mu-1$ for
any horizontal derivative $W$;
\item[iii)]  If $\mu>0$, then $K\in L^1_{\mathrm{loc}}(\G)$.
\end{itemize}
\end{proposition}

\begin{lemma}\label{closed}
 Let $g$ be a a kernel of type $\mu>0$,
and let $\psi\in \mc D(\G)$ be a test function.
Then $\psi \ast g $ is smooth on $\G$.
 
If, in addition, $R$ is an homogeneous
 polynomial of degree $\ell\ge 0$ in the horizontal derivatives,
 we have
 $$
R( \psi\ast g)(p)= O(|p|^{\mu-Q-\ell})\quad\mbox{as }p\to\infty.
 $$
 On the other hand, if  $g$ is
 a smooth function in $\G\setminus\{0\}$ that
 satisfies the logarithmic estimate
$|g(p)|\le C(1+|\ln|p|| )$ and in addition
its horizontal derivatives are homogeneous of degree $-1$
with respect to group dilations,
 then, if $\psi \in \mc D(\G)$ and $R$ is an homogeneous
 polynomial of degree $\ell\ge 0$ in the horizontal derivatives,
 we have
 \begin{eqnarray*}
R( \psi \ast g)(p)&=&O(|p|^{-\ell})\quad \mbox{as }p\to\infty
\quad\mbox{ if $\ell>0$}; \\
R(\psi\ast g)(p)&=&O(\ln|p| )\quad \mbox{as }p\to\infty
\quad\mbox{ if $\ell=0$}.
\end{eqnarray*}
\end{lemma}

In particular, if $\psi\in \mc D(\G)$, and $K$ is a kernel of type
$\mu <Q$, then both $\psi \ast K$
and all its derivatives belong to $L^\infty(\G)$.

In the following theorem we gather some continuity properties for convolutions that can be find in \cite{folland} and \cite{folland_stein}
(or easily derived from \cite{folland}  \cite{folland_stein}).

\begin{theorem}  \label{folland cont} We have:
\begin{itemize}
         \item[i)] Hausdorff-Young inequality holds, i.e.,  if $f\in L^p(\G)$, $g\in L^q(\G)$, $1\le p, q,r \le \infty$ and $\frac1p + \frac1q - 1 = \frac1r$, then $f\ast g\in L^r(\G)$
         (see \cite{folland_stein}, Proposition 1.18) .
	\item[ii)] If $K$ is a kernel of type $0$, $1<p<\infty$, $\ge 0$, then the mapping $T:u\to u\ast K$ defined for $u\in\mc D(\G)$ extends to a bounded operator on $W^{s,p}(\G)$
	(see \cite{folland}, Theorem 4.9).
	\item[iii)] Suppose $0<\mu <Q$, $1<p<Q/\mu$ and $\frac{1}{q}=\frac{1}{p}-\frac{\mu}{Q}$. Let $K$ be a kernel of type $\mu$. If $u\in L^p(\G)$ the convolutions $u\ast K$ and $K\ast u$ exists a.e. and are in $L^q(\G)$ and there is a constant $C_p>0$ such that 
	$$
	\|u\ast K\|_q\le C_p\|u\|_p\quad \mathrm{and}\quad \| K\ast u\|_q\le C_p\|u\|_p\,
	$$
	(see \cite{folland}, Proposition 1.11).
	\item[iv)] Suppose  $s\ge 1$, $1<p<Q$, and let $\mc U$ be a bounded open set. If $K$ is a kernel of type $1$ and $u\in W^{s-1,p}(\G)$  with $\supp u \subset \mc U$, then 
	$$
	\|u\ast K\|_{W^{s,p}(\G)} \le C_{\mc U} \|u\|_{W^{s-1,p}}(\G).
	$$
\end{itemize} 
\end{theorem}

\begin{proof} The proof of iv) can be carried out relying on Theorems 4.10, 4.9 and Proposition 1.11 of \cite{folland}, keeping into account that
$L^{pQ/(Q-p)}(\mc U) \subset L^{p}$ and Proposition \ref{kernel}, ii). Indeed
\begin{equation*}\begin{split}
\|u\ast K\|_{W^{s,p}(\G)} &\le C\big\{ \|u\ast K\|_{L^p (\G)} +\sum_{\ell =1}^m \|u\ast W_\ell K\|_{W^{s-1,p}(\G)} \big\}
\\&
 \le C\big\{ \|u\ast K\|_{L^p (\G)} + \|u\|_{W^{s-1,p}(\G)} \big\}
\\&
\le 
C\big\{ \|u\|_{L^{pQ/(Q-p}(\G)} + \|u\|_{W^{s-1,p}(\G)} \big\}
\le
C_{\mc U} \|u\|_{W^{s-1,p}}(\G).
\end{split}\end{equation*}

\end{proof}

\begin{definition} Let $f$ be a measurable function on $\G$. If $t>0$ we set
$$
\lambda_f(t) = |\{|f|>t\}|.
$$
If $1\le p\le\infty$ and
$$
 \sup_{t>0} \lambda_f^p(t)  <\infty,
$$
we say that $f\in L^{p,\infty}(\G)$.
\end{definition}

\begin{definition}\label{M}
Following \cite{BBD}, Definition A.1, if $1<p<\infty$, we set 
$$
\| u\|_{M^p} : = \inf \{C\ge 0 \, ; \, \int_K |u| \, dx \le C |K|^{1/p'}\;
\mbox{for all $L$-measurable set $K\subset \G$}\}.
$$
and $M^p = M^p(\G)$ is the set of measurable functions $u$ on $\G$ satisfying $\| u\|_{M^p}<\infty$.
\end{definition}

Repeating verbatim the arguments of \cite{BBD}, Lemma A.2, we obtain

\begin{lemma}\label{marc alternative} If $1<p<\infty$, then
$$
\dfrac{(p-1)^p}{p^{p+1}}  \| u\|_{M^p }^p  \le \sup_{\lambda >0} \{\lambda^p | \{|u|>\lambda\} |\, \} \le  \| u\|_{M^p } ^p.
$$
In particular, if $1<p<\infty$, then  $M^p  = L^{p,\infty}(\G)$.
\end{lemma}

\begin{corollary}\label{marc alternative coroll}  If $1\le s <p$, then $M^p \subset L^s_{\mathrm{loc}} (\G)\subset L^1_{\mathrm{loc}} (\G)$.

\end{corollary}

\begin{proof} By Lemma \ref{marc alternative}, if $u\in M^p$ then $|u|^s\in M^{p/s}$, and we can conclude
thanks to Definition \ref{M}.

\end{proof}

\begin{lemma}\label{convolutions} Let $E$ be a kernel of type $\alpha\in (0,Q)$. Then for all $f\in L^1(\G)$ we have $f\ast E\in M^{Q/(Q-\alpha)} $
and there exists $C>0$ such that 
$$
 \| f\ast E\|_{M^{Q/(Q-\alpha)}} \le C\|f \|_{L^1({\G})}
   $$
for all $f\in L^1(\G)$. In particular, by Corollary \ref{marc alternative coroll}, 
if $1\le p < Q/(Q-\alpha)$, then $f\ast E\in
L^{p}_{\mathrm{loc}}(\G) \subset  L^1_{\mathrm{loc}}(\G)$.
\end{lemma}

As in \cite{BFP2}, Lemma 4.4 and Remark 4.5, we have:

\begin{remark}\label{truncation} Suppose $0< \alpha<Q$.
If $K$ is a kernel of type $\alpha$
and $\psi \in \mc D(\G)$, $\psi\equiv 1$ in a neighborhood of the origin, then
the statements of Lemma \ref{convolutions} still
hold if we replace $K$ by $(1-\psi )K$ or by $\psi K$. 

\end{remark}

\subsection{Estimates on shells}

Here, we prove a fine boundedness property of kernels in $L^1$, expressed in terms of $L^1$ norms on shells. It will play a crucial role in section \ref{GNI}. We start with a preliminary duality lemma.
  
\begin{lemma}\label{closed ex bis}
If  $K$ is a kernel of type $\mu\in (0,Q)$,
$u\in L^1(\G)$ and
$\psi\in \mc D(\G)$, then
\begin{equation}
\Scal{u\ast K}{ \psi} = \Scal{ u}{\psi\ast\ccheck K}.
\end{equation}
\end{lemma}
In this equation, the left hand side is the action of a matrix-valued distribution on a vector-valued test function, see formula \eqref{matrix form}, the right hand side is the inner product of an $L^1$ vector-valued function with an $L^\infty$ vector-valued function.

\begin{proof}

The assertion follows by Fubini-Tonelli theorem.
 Indeed
\begin{equation}\label{jan 8 eq:1}\begin{split}
\int\int  & |K(y^{-1}x)|\, | u (y) | | \psi (x) |\, dy\, dx 
\\&
\\& \le  
C\,  \int\int  d(x,y)^{\mu-Q} \, | u (y) | | \psi(x) |\, dy\, dx 
< \infty.
\end{split}\end{equation}

 Since $\psi$ is compactly supported there exists $M>0$ such that the above integral
can be written as
$$
\int_{|x|\le M}\int \cdots = \int_{|x|\le M}\int_{|y|\le 2M} \cdots 
+
 \int_{|x|\le M}\int_{|y|> 2M} \cdots.
$$ 
Now
\begin{equation}\begin{split}
 \int_{|x|\le M} & \int_{|y|\le2M} d(x,y)^{\mu-Q} \, | u (y) | | \psi  (x) |\, dy\, dx
 \\&
 \le C_{\psi }  \int_{|x|\le M} \int_{|y|\le 2M} d(x,y)^{\mu-Q} \, | u (y) |\, dy, dx
  \\&
 \le C_{\psi }  \int_{|x|\le M}\Big( \int_{d(x,y)\le 3M} d(x,y)^{\mu-Q} \,  dx\Big) | u (y) |\, dy
 \\&
 \le C_{\psi } \|u\|_{L^1{(\G)}}.
\end{split}\end{equation}
On the other hand, if $|x|\le M$ and $|y|> 2M$, then $d(x,y) >M$. so that
\begin{equation}\begin{split}
 \int_{|x|\le M}\int_{|y|> 2M} & d(x,y)^{\mu-Q} \, | u (y) | | \psi  (x) |\, dy\, dx
 \\&
 \le M^{\mu-Q} \|\psi \|_{L^1{(\G)}}  \|u\|_{L^1{(\G)}}.
\end{split}\end{equation}
Then
\begin{equation}\begin{split}
\int \Big(\int &  K(y^{-1}x) u (y)\, dy\Big)\,  \psi  (x) \, dx
\\&
= \int \Big(\int K(y^{-1}x) \psi  (x) \,dx \Big)  u (y)\, dy
\\&
= \int \Big(\int \ccheck K(x^{-1}y) \psi  (x) \, dx \Big)  u (y)\, dy,
\end{split}\end{equation}
and therefore 
we are done.

\end{proof}

{  \begin{remark}\label{closed ex bis remark} The conclusion of Lemma \ref{closed ex bis} still holds
if we assume $K\in L^1_{\mathrm{loc}}(\G)$, provided $u$ is compactly
supported.
\end{remark}
}

\begin{theorem}\label{anuli}
If $K$ is a kernel of type $\alpha\in (0,Q)$, then for any $f\in L^1(\G)$ such that
\begin{equation}
	\label{average}
\int_{\G} f(y)\, dy = 0,
\end{equation}
we have:
$$
R^{-\alpha} \int_{B(e,2R)\setminus B(e,R)} |K\ast f| \, dx \longrightarrow 0\qquad\mbox{as $R\to\infty$.}
$$

\end{theorem}
\begin{proof}
If $R>1$, taking into account \eqref{average}, we have:
\begin{equation*}\begin{split}
R^{-\alpha}  & \int_{B(e,2R)\setminus B(e,R)} |K\ast f| \, dx
=
R^{-\alpha}\int_{R<|x|<2R} \,dx \Big| \int  \,  K(y^{-1}x) f(y) \, dy\Big|
\\&
= R^{-\alpha}\int_{R<|x|<2R} \,dx \Big| \int  \, \big[ K(y^{-1}x)-K(x)\big] f(y) \, dy\Big|
\\&
\le
R^{-\alpha}  \int  \,  | f(y)| \Big(  \int_{R<|x|<2R}  \Big| K(y^{-1}x)-K(x) \Big|  \,dx \Big)  \, dy
\\&
=R^{-\alpha} \int_{|y|<\frac12 R} \,  | f(y)| \big( \cdots \big)  \, dy
+
R^{-\alpha} \int_{ \frac12 R<|y| <4R} \,  | f(y)| \big( \cdots \big)  \, dy
\\&
+
R^{-\alpha} \int_{|y| > 4 R} \,  | f(y)| \big( \cdots \big)  \, dy
\\&
=: R^{-\alpha} I_{1}(R) + R^{-\alpha} I_{2}(R) + R^{-\alpha} I_{3}(R).
\end{split}\end{equation*}
Consider first the third term above. By homogeneity we have
$$
 I_{3}(R) \le C_K\,  \int_{|y|>4 R} \,  | f(y)| \big( \int_{R<|x|<2R} (d(x,y)^{-Q+\alpha} + d(x,e)^{-Q+\alpha}  ) \,dx  \big)  \, dy
$$
Notice now that, if $|y| >4R$ and $ R<|x|<2R$, then $d(x,y)\ge |y| - |x|
\ge 4 R -R = \ge \frac{3}2 |x|$. Therefore, by \cite{folland_stein}, Corollary 1.16,
\begin{equation*}\begin{split}
d(x,y)^{-Q+\alpha} & + d(x,e)^{-Q+\alpha} \le \left\{\big( \dfrac{2}{3}\big)^{Q-\alpha} + 1\right\} |x|^{-Q+\alpha},
\end{split}\end{equation*}
and then
$$
 \int_{R<|x|<2R} (d(x,y)^{-Q+\alpha} + d(x,e)^{-Q+\alpha}  ) \,dx \le C_\alpha \, R^\alpha.
$$
Thus
$$
R^{-\alpha} I_{3}(R) \le C_{K,\alpha}\,  \int_{|y|>4R} \,  | f(y)|   \, dy \longrightarrow 0
$$
as $R\to\infty$.

Consider now the second term. Again we have
$$
 I_{2}(R) \le C_K\,  \int_{\frac12 R<|y|<4 R} \,  | f(y)| \big( \int_{R<|x|<2R} (d(x,y)^{-Q+\alpha} + d(x,e)^{-Q+\alpha}  ) \,dx  \big)  \, dy.
$$
Obviously, as above,
$$
\int_{R<|x|<2R} d(x,e)^{-Q+\alpha} \,dx \le C R^\alpha.
$$
Notice now that, if $\dfrac12 R < |y| <4R$ and $ R<|x|<2R$, then $d(x,y)\le |x| + |y|
\le 6R$. Hence
$$
 \int_{\frac12 R<|y|<4 R} \,  | f(y)| \big( \int_{d(x,y)<6R} d(x,y)^{-Q+\alpha} \,dx  \big)  \, dy \le CR^\alpha.
$$
Therefore 
$$
R^{-\alpha}I_{2}(R) \le C_K\,  \int_{\frac12 R<|y|<4 R} \,  | f(y)|\, dy \longrightarrow 0
$$
as $R\to\infty$.
Finally, if $|y| < \frac{R}2$ and $ R<|x|<2R$ we have $|y| < \frac12 |x|$, so that, by \cite{folland_stein},
Proposition 1.7 and Corollary 1.16,
\begin{equation*}\begin{split}
R^{-\alpha} I_{1}(R) & \le C_K\,  \int_{|y|<\frac12 R} \,  | f(y)| \big( \int_{R<|x|<2R} \frac{|y|}{|x|^{Q-\alpha+1}} \,dx  \big)  \, dy
 \\&
 = C_K\,  \int_{\G} \,  | f(y)||y| \chi_{[0,\frac12 R]}(|y|) \big(R^{-\alpha} \int_{R<|x|<2R} \frac{1}{|x|^{Q-\alpha+1}} \,dx  \big)  \, dy
  \\&
\le C_K\,  \int_{\G} \,  | f(y)||y| \chi_{[0,\frac12 R]}(|y|)R^{-1} \, dy=: C_K\,  \int_{\G} \,  | f(y)|H_R(|y|) \, dy.
\end{split}\end{equation*}
Obviously, for any fixed $y\in \G $ we have $(|y|) H_R (|y|)\to 0$ as $R\to\infty$. On the other hand, 
$ | f(y)|H_R(|y|) \le \frac12 |f(y)|$, so that, by dominated convergence theorem,
$$
R^{-\alpha} I_{1}(R) \longrightarrow 0
$$
as $R\to\infty$.

This completes the proof of Theorem \ref{anuli}.
\end{proof}

\subsection{Powers of Kohn's Laplacian and Sobolev spaces}\label{7.3}

In section \ref{SH}, we shall construct operators of order $-1$, and we shall need to show that they improve differentiability. They win one degree of differentiability on the $L^p$ scale when $p>1$, but not on the $L^1$ scale. This is why we need introduce fractional Sobolev spaces, fortunately only for exponents $p>1$. We choose to define them using powers of Kohn's Laplacian.

Let $\{X_1,\dots,X_m\}$ be the fixed basis of the horizontal layer $V_1$
of $\mathfrak g$ chosen above. 

 We denote by $\Delta_\G$ the negative
horizontal sublaplacian 
$$
\Delta_\G := \sum_{j=1}^mX_j^2.
$$
If $1<p<\infty$ and $a\in\mathbb C$, we define $(-\Delta_\G)^{a/2}$ in $ L^p(\G)$ following \cite{folland}.
If in addition $s\ge 0$, again as in \cite{folland}, we denote by $W^{s,p}_\G(\G)$
the domain of the realization of $(-\Delta_\G)^{s/2}$ in $L^p(\G)$ endowed with the
graph norm. In fact, as soon as $p\in (1,\infty)$ is fixed, to avoid
cumbersome notations, we do not stress the explicit dependence on $p$ of the
fractional powers $(-\Delta_\G)^{s/2}$ and of its domain.

\begin{remark}\label{interpolation remark} By \cite{saka}, Proposition 6, if $p>1$, then
the spaces $W^{s,p}_\G(\G)$, $s\ge 0$ provide a complex interpolation scale
of Banach spaces (see e.g. \cite{berg_lofstrom}).

\end{remark}

\begin{proposition}\label{invariance} The operators $(-\Delta_\G)^{s/2}$ are
left invariant on $W_\G^{s,p}(\G)$.
\end{proposition}

We recall that

\begin{proposition}[\cite{folland}, Corollary 4.13]\label{integer spaces} If $1<p<\infty$
and $\ell\in\mathbb N$, then
the space $W_\G^{\ell,p}(\G)$
coincides with the space of all $u\in L^p(\G)$ such that
$$
X^I u\in L^p(\G)\quad\text{for all multi-indices }I\text{ with } d(I)=\ell,
$$
endowed with the natural norm.

\end{proposition}

\begin{proposition}[\cite{folland}, Corollary 4.14] If $1<p<\infty$
and $s\ge 0$, then the space $W_\G^{s,p}(\G)$
is independent of the choice of $X_1,\dots,X_m$.
\end{proposition}

\begin{proposition}\label{density0} If $1<p<\infty$
and $s\ge 0$, then $ \mc S(\G)$
and $\mc D(\G)$ are dense subspaces of $W_\G^{s,p}(\G)$.
\end{proposition}

\begin{theorem}[\cite{folland}, Corollary 4.15]\label{berthollet 2}
If $\phi\in \mc D(\G)$, the map $f\to \phi f$ is continuous from $W^{s,p}(\G)$
to  $W^{s,p}(\G)$ for $p> 1$ and $s\ge 0$.
\end{theorem}

The following Proposition is a tool to prove that a given operator maps a suitable function space into a Sobolev space $W^{s,p}$. Indeed, it reduces the question to the case of the kernel of a negative power of $\Delta_\G$. It will be used in Lemma \ref{berthollet 1}.

  \begin{proposition}[see \cite{folland}]\label{pdalpha}
  Suppose  $0<\beta<Q$.   Denote by $h=h(t,x)$ the fundamental
solution of $-\Delta_\G +\partial/\partial t$ (see \cite{folland}, Proposition 3.3). Then the integral
$$
R_\beta(x) =\frac{1}{\Gamma(\beta/2)}
\int_0^{\infty}t^{\frac{\beta}{2}-1}h(t,x)\, dt
$$
converges absolutely for $x\neq 0$.

Moreover
\begin{itemize}
\item [i)]  $R_\beta$ is a kernel of type $\beta$;
\item[ii)] if $\alpha\in (0,2)$ and $u\in\mc D(\G)$, then
$$
(-\Delta_\G)^{\alpha/2} u = -\Delta_\G (u\ast R_{2-\alpha}).
$$
\end{itemize}
 \end{proposition}

\subsection{Function spaces in domains}

When dealing with subRiemannian manifolds in section \ref{SH}, we shall need to localize Sobolev spaces on balls and transport them by contactomorphims. Therefore we provide a precise definitions of $W^{s,p}(D)$ for $D$ a good domain, typically a ball, in a Carnot group.
 
\begin{definition}\label{integer spaces bis}
As in Proposition \ref{integer spaces}, if $D \subset \G$ is a connected open set, $\ell$ is a nonnegative integer 
and $p\ge 1$, we set
 $$
 W^{\ell,p}(D) := \{u\in L^p(D)\; : \, W^I u\in L^p(D) \,  , \, d(I)\le \ell\}.
 $$
\end{definition}

From now on, we assume  that $D$ is an {\sl extension domain}, i.e. 
\begin{definition}\label{extension domain} We say that a connected bounded open set $D\subset\G$ is 
an {\sl extension domain} if  it enjoys the so-called
 \emph {extension property}, i.e. for any $\ell\in\N$ there exists a bounded linear operator
\begin{equation}\label{extension}
p_\ell :  W^{\ell,p}(D)  \to  W^{\ell,p}(\G) 
\end{equation}
 such that $p_\ell u \equiv u$ in $D$.
 \end{definition}
 
 Sufficient conditions yielding that $D$ enjoys the 
extension property are largely studied in the literature.  We do
not enter  into technical details, but we recall the following facts:
\begin{itemize}
\item In general Carnot groups  ``elementary'' qualitative conditions for
\eqref{extension} are not known. Smooth domains may fail to be extension domains.
\item The so-called $(\eps,\delta)$ (or uniform) domains are extension domains.
In particular, in Heisenberg groups, Carnot-Carath\'eodory balls are
extension domains.
\item In Carnot groups of step 2,  $C^{1,1}$-domains are
extension domains. In particular, we shall need later that Kor\'anyi balls in Heisenberg groups
(see \eqref{gauge} below) are extension domains. In particular, in Heisenberg groups
there is a basis of the topology made by extension domains. This provides a precise meaning
for the fractional local Sobolev spaces $W^{\ell,p}_{\mathrm{loc}}(\G)$.
\item In general Carnot groups, bounded intrinsic Lipschitz domains are extension domains.

\end{itemize}
For proofs of the above results and for an overview of the problem we refer
for instance to \cite{CGN_novosibirsk}, \cite{lu_acta_sinica}, \cite{nhieu_AMPA}, \cite{GN}, \cite{CG}, \cite{MoMo}
\cite{monti_thesis}, \cite{VG}, \cite{FPS}.

The following definition is not optimal but suffices for our purposes: {if $s$ is a nonnegative integer, the notion of Sobolev space given in 
Definitions \ref{integer spaces bis} is  equivalent to the following one when the domain is an extension domain, as we will see in Remark \ref{integer sobolev in domains} below}.

\begin{definition}\label{fractional sobolev in domains}  Let $D$ be a connected bounded open set enjoying the extension property
\eqref{extension}.  Denote by $r_D$ the operator of restriction to $D$. If $s\ge 0$ and $p>1$ we set
$$
W^{s,p}(D) = \{ r_D u\, , u\in W^{s,p}(\G)\},
$$
endowed with the norm
\begin{equation}\label{fractional sobolev in domains norm}
\|u\|_{W^{s,p}(D)} := \inf \{ \|v\|_{W^{s,p}(\G)}\, , r_Dv =u\}.
\end{equation}

\end{definition}

\begin{remark}\label{integer sobolev in domains} If $s$ is a nonnegative integer, then
Definitions \ref{fractional sobolev in domains} and \ref{integer spaces bis} are equivalent
(in bounded extension domains). Indeed, denote for a while by $\| u\|_{W^{s,p}(D)}^*$  the norm
defined in  \eqref{fractional sobolev in domains norm}, keeping the notation $\| u\|_{W^{s,p}(D)}$
for the norm of Definition  \ref{integer spaces bis}. 
Thus,
since $r_D p_s u \equiv u$, we have
\begin{equation*}\begin{split}
\| u\|_{W^{s,p}(D)}^* & \le \|p_s u\|_{W^{s,p}(\G)} \le C \| u\|_{W^{s,p}(D)}.
\end{split}\end{equation*}
On the other hand, let $v$ be an arbitrary extension
of $u$ outside  $D$. We notice that for any horizontal derivative $W $ we have $Wu = r_D(Wv)$.
Thus
$$
 \| u\|_{W^{s,p}(D)} \le  \| v\|_{W^{s,p}(\G)},
$$
so that, taking the infimum of the right-hand side of this inequality for all extensions $v$, we have
$$
 \| u\|_{W^{s,p}(D)} \le  \| u\|_{W^{s,p}(D)} ^*.
$$

\end{remark}

\begin{remark} It is easy to see that Proposition \ref{density0} and Theorem \ref{berthollet 2}
still hold for Sobolev spaces in $D$.
\end{remark}

\subsection{Truncated kernels}

The interior inequalities of Theorem \ref{poincareintro} rely on convolution with functions of the form $\psi K$ where $K$ is a kernel and $\psi$ a smooth cut-off. We establish now boundedness properties of such operators.

\begin{lemma}\label{berthollet 1}
Let $K$ be a kernel of type $\alpha\in (0,2]$ and let $\psi\in \mc D(\G)$, $\psi\equiv 1$ in a neighborhood of the origin. Let $\chi_0\in\mc D(\G)$.
If  $D$ is a bounded extension domain (see Definition \ref{extension domain}), $\alpha ' >0$, $\alpha - 1<\alpha' <\alpha$ and 
$Q/(Q-\alpha)>p>Q/(Q-\alpha+\alpha')>1$, 
then  the map
$$
\chi_0f\to (\chi_0f)\ast \psi K
$$
is continuous from $L^1(\G)$ to $W^{\alpha',p}(D)$.   
\end{lemma}

\begin{proof} 

Since both  $\chi_0$ and $\psi K$ are compactly supported, then $(\chi_0f) \ast \psi K$
is compactly supported in a bounded open set $\mc U$ and, obviously, is an extension of
$r_D((\chi_0f) \ast \psi K)$. Hence
$$
\| (\chi_0f) \ast \psi K\|_{W^{1,\alpha'}(D)}\le \| (\chi_0f) \ast \psi K\|_{L^p(\mc U)} + \| \Delta_\G^{\alpha'/2}((\chi_0f)\ast \psi K)\|_{L^p(\G)}.
$$
Since $p< Q/(Q-\alpha)$, by Theorem \ref{folland cont} - i),
$ f \ast \psi K $ belongs to $L^p(\G)$ and the first term above
is estimated as we want.

We are left with the estimation of the second norm above.

By density we can always suppose $f\in\mc D(\G)$, so that $(\chi_0f)\ast \psi K\in \mc D(\G) $. Thus, by Proposition \ref{pdalpha}, ii)
\begin{equation}\label{feb1:1}
-\Delta_\G^{\alpha'/2}((\chi_0f)\ast \psi K) = -\Delta_\G \big(((\chi_0f)\ast \psi K)\ast R_{2-\alpha'}\big).
\end{equation}
On the other hand, we can write $(\chi_0f)  \ast \psi K = (\chi_0f)\ast K - (\chi_0f)\ast (1-\psi)K$,
so that
\begin{equation}\label{70}
-\Delta_\G^{\alpha'/2}((\chi_0f)\ast \psi K) = -\Delta_\G^{\alpha'/2}((\chi_0f)\ast K) - -\Delta_\G^{\alpha'/2}((\chi_0f)\ast (1-\psi )K).
\end{equation}

Let us consider the second term of \eqref{70}.  We notice first that,
keeping in mind that $R_{2-\alpha'}$ is a kernel of type 
 ${2-\alpha'}$, we can apply Lemma 1.12  of \cite{folland}, to get
\begin{equation}\label{associative}
(f\ast (1-\psi )K)\ast R_{2-\alpha'} = f \ast \big((1-\psi )K \ast R_{2-\alpha'}\big).
\end{equation}
Indeed, 
$K$ is a kernel of type $\alpha$, and then
$$
|(1-\psi)K| \le C \big(1+ |x|^{\alpha - Q}\big),
$$
so that $(1-\psi)K \in L^q(\G)$ for fome suitable $q>1$, provided
$1/q < 1-\alpha/Q$.
In addition
$$
1+ \frac1q - \frac{2-\alpha'}{Q} - 1 >0,
$$
since
$$
\frac{Q-\alpha}{Q} - \frac{2-\alpha'}{Q} >0,
$$
and we can alway choose $q$ such that
\begin{equation}\label{q}
\frac1q \in \big( \frac{2-\alpha'}{Q}, \frac{Q-\alpha}{Q} \big).
\end{equation}
This prove \eqref{associative}.

We stress that the choice of $q$ will not affect the remaining part of the proof,
since $q$ is merely a tool used to prove identity \eqref{associative}.

By \eqref{associative},
we get
\begin{equation} \label{feb1:2}\begin{split}
\Delta_\G & \big(((\chi_0f)\ast (1-\psi)K)\ast R_{2-\alpha'}\big) = 
 (\chi_0f) \ast  \Delta_\G\big((1-\psi )K \ast R_{2-\alpha'}\big)
\\& =
 (\chi_0f) \ast  \big( \ccheck \Delta_\G ((1-\psi )\ccheck K) \ast R_{2-\alpha'}\big) .
\end{split}\end{equation}
Take now $s>1$ such that 
\begin{equation}\label{s}
\frac1s =\frac1p + \frac{2-\alpha'}{Q}.
\end{equation}
Keeping into account  that $\ccheck K$ is still a kernel of type $\alpha$ and that
$1-\psi\equiv 1$ near the infinity, by Lemma \ref{closed} we have
$$
|\ccheck \Delta_\G ((1-\psi )\ccheck K)| \le C \big(1+ |x|^{\alpha - Q-2}\big).
$$
Hence  $\ccheck \Delta_\G ((1-\psi )\ccheck K)  \in L^s(\G)$.
 Therefore, by Theorem \ref{folland cont} - iii) and \eqref{s}
$$
\Delta_\G\big((1-\psi )K \ast R_{2-\alpha'}\big) \in L^{p}(\G).
$$

Combining \eqref{feb1:1} and \eqref{feb1:2}, by Hausdorff-Young theorem (see Theorem \ref{folland cont} -i)) we have
\begin{equation*}
 \| \Delta_\G^{\alpha'/2}((\chi_0f)\ast (1-\psi)  K)\|_{L^p(\G)}
 \le
 \| f\|_{L^1(\G)} \| \Delta_\G\big((1-\psi )K \ast R_{2-\alpha'}\big)\|_{L^p(\G)}.
\end{equation*}
This provides an estimate of the second term of \eqref{70}.

Thus, we have but to consider the term $-\Delta_\G \big(((\chi_0f)\ast K)\ast R_{2-\alpha'}\big)$.
Since $\alpha + 2-\alpha' <3\le Q$, by Proposition 1.13 of \cite{folland},
\begin{equation*}\begin{split}
-\Delta_\G & \big(((\chi_0f)\ast K)\ast R_{2-\alpha'}\big) = -\Delta_\G \big(((\chi_0f)\ast (K\ast R_{2-\alpha'})\big)
\\&=
 - (\chi_0f)\ast \big(\Delta_\G (K\ast R_{2-\alpha'})\big),
\end{split}\end{equation*}
where $K\ast R_{2-\alpha'}$ is a kernel of type $\alpha + 2-\alpha'$, so that, by
Proposition \ref{kernel}, $\Delta_\G (K\ast R_{2-\alpha'})$ is a kernel of type
$\alpha - \alpha'$. Therefore, by Lemma \ref{convolutions}, 
$$
\|  (\chi_0f)\ast \big(\Delta_\G (K\ast R_{2-\alpha'})\big)\|_{L^p(D)} \le
C\| f\|_{L^1(\G)}
$$
and the proof is completed.
\end{proof}

 The proof of the following result is similar to the previous one but for sake of completeness we write down the details.

{

\begin{theorem} \label{berthollet 3} Suppose $p>1$,  and let $\chi_0\in\mc D(\G)$ be fixed.
\begin{itemize}
\item[i)]Let $K$ be a kernel of type 1
and let  $\psi \in \mc D(\G)$ be as in Remark \ref{truncation}  above, i.e. $\psi\equiv 1$ in a neighborhood of the origin. 
In addition, let $D\subset \G$ be a bounded connected extension domain.
Then the map $f\to (\chi_0f) \ast \psi K$
is continuous from $W^{s-1, p} (\G)$ to $W^{s, p} (D)$
for $s\ge 1$.

\item[ii)] Analogously, if $K$ is a kernel of type 0, then the map $(\chi_0f)\to f\ast \psi K$
is continuous from $W^{s, p} (\G)$ to $W^{s, p} (D)$.
\end{itemize}
for $s\ge 0$. 

\end{theorem}

\begin{proof} Since both  $\chi_0$ and $\psi K$ are compactly supported, then $(\chi_0f) \ast \psi K$
is compactly supported. 

Let  now $\psi_0$ be a cut-off function, $\psi_0\equiv 1$ on $D$,
so that $ \psi_0\big\{  (\chi_0f) \ast \psi K \big\} $ is an extension of $  r_D \big\{  (\chi_0f) \ast \psi K \big\} $.

Then, by definition (see \eqref{fractional sobolev in domains norm}), 
$$
\|   (\chi_0f) \ast \psi K   \|_{W^{s, p} (D)} \le \| \psi_0\big\{ (\chi_0f) \ast \psi K\big \}  \|_{W^{s, p} (\G)}.
$$
 Therefore, we have but to prove that
\begin{equation}\label{to prove 3/1}
\| \psi_0\big\{ (\chi_0f) \ast \psi K\big \}  \|_{W^{s, p} (\G)} \le C \|  f \|_{W^{s-1, p} (\G)}.
\end{equation}

By density  (see Proposition \ref{density0}), we can assume $f\in\mc D(\G)$. In addition,
by interpolation (see Remark \ref{interpolation remark}), we can assume $s$ integer.  As in the proof
of Lemma \ref{berthollet 1} we write $\psi K = K -(1-\psi) K$.   
By Theorem \ref{folland cont}, iv), 
\begin{equation*}
\| \psi_0\big\{ (\chi_0f) \ast K\big \}  \|_{W^{s, p} (\G)}
\le
 \| (\chi_0f) \ast  K \|_{W^{s, p} (\G)}
\le  \| \chi_0f \|_{W^{s-1, p} (\G)} \le  \| f \|_{W^{s-1, p} (\G)}.
\end{equation*}
On the other hand, the $W^{s, p}$-norm of  $\psi_0\big\{ (\chi_0f) \ast (1-\psi)K\big \}$ can be estimated by a sum of terms of
the form
\begin{equation*}\begin{split}
\int_{\G}  & |W^J\psi_0|^p |(\chi_0f) \ast  W^I((1-\psi) K|^p \, dx
\le
C\int_{\supp \psi_0}  \!\!\!  \!\!\!\!\!\!  \!\!\! |(\chi_0f) \ast  W^I((1-\psi) K|^p \, dx
\end{split}\end{equation*}
with $d(I)+d(J)\le s$.

We notice now that $(1-\psi)K$ is a smooth function supported away from the origin. Therefore,
keeping into account that $\chi_0$ and $\psi_0$ are compactly supported,
we can assume that $ (1-\psi)K$ is compactly supported away from the origin, so that
$ W^I((1-\psi) K$ belongs to  $L^1(\G)$.
Thus eventually, once more by Hausdorff-Young inequality (Theorem \ref{folland cont})
$$
\Big(\int_{\supp \psi_0}  \!\!\!  \!\!\!\!\!\!  \!\!\! |(\chi_0f) \ast  W^I((1-\psi) K|^p \, dx\Big)^{1/p} \le C \|\chi_0f\|_{L^p(\G)}
\le C \|f\|_{L^p(\G)}\le  \| f \|_{W^{s-1, p} (\G)}.
$$
This completes the proof of assertion i).

The proof of assertion ii) can be carried out in the same way, using 
Theorem \ref{folland cont}, ii) instead of Theorem \ref{folland cont}, iv).

\end{proof}

}

\section{Preliminary results on Heisenberg groups, Rumin's complex and Laplacians}\label{heisenberg}

\subsection{Heisenberg groups}

The $n$-dimensional Heisenberg group $\he n$ is the $2$-step Carnot group whose Lie algebra
\begin{equation*}
\mfrak h=\mfrak h_1\oplus \mfrak h_2.
\end{equation*}
has $\mfrak h_2=\R$, $\mfrak h_1=\R^{2n}$ and the Lie bracket $\mfrak h_1\times\mfrak h_1\to\mfrak h_2$ is a nondegenerate skew-symmetric form.

The group can be identified with $\rn {2n+1}$ through exponential
coordinates and a point $p\in \he n$ can be denoted by
$p=(x,y,t)$, with both $x,y\in\rn{n}$
and $t\in\R$. For a general review on Heisenberg groups and their properties, we
refer to \cite{Stein}, \cite{GromovCC} and to \cite{VarSalCou}.
We limit ourselves to fix some notations following \cite{BFP2}.

The Heisenberg group $\he n$ can be endowed with the homogeneous
norm (Kor\'anyi norm)
\begin{equation}\label{gauge}
\varrho (p)=\big((|x|^2+|y|^2)^2+t^2\big)^{1/4},
\end{equation}
and we define the gauge distance (a true distance, see
 \cite{Stein}, p.\,638, that is equivalent to
Carnot--Carath\'eodory distance)
as $d(p,q):=\varrho ({p^{-1}\cdot q}).$
Finally, set $B(p,r)=\{q \in  \he n; \; d(p,q)< r\}$.

The standard basis of $\mfrak
h$ is given, for $i=1,\dots,n$,  by
\begin{equation*}
X_i := \partial_{x_i}-\frac12 y_i \partial_{t},\quad Y_i :=
\partial_{y_i}+\frac12 x_i \partial_{t},\quad T :=
\partial_{t}.
\end{equation*}
The only non-trivial commutation  relations are $
[X_{j},Y_{j}] = T $, for $j=1,\dots,n.$ 

The vector space $ \mfrak h$  can be
endowed with an inner product, indicated by
$\scalp{\cdot}{\cdot}{} $,  making
    $X_1,\dots,X_n$,  $Y_1,\dots,Y_n$ and $ T$ orthonormal.
    
Throughout this paper, we write also
\begin{equation}\label{campi W}
W_i:=X_i, \quad W_{i+n}:= Y_i, \quad W_{2n+1}:= T, \quad \text
{for }i =1, \cdots, n.
\end{equation}

 The dual space of $\mfrak h$ is denoted by $\covH 1$.  The  basis of
$\covH 1$,  dual to  the basis $\{X_1,\dots , Y_n,T\}$,  is the family of
covectors $\{dx_1,\dots, dx_{n},dy_1,\dots, dy_n,\theta\}$ where 
$$ 
\theta:= dt - \frac12 \sum_{j=1}^n (x_jdy_j-y_jdx_j)
$$
is called the {\it contact form} in $\he n$. A diffeomorphism $\phi$ between open subsets of $\he{n}$ is called a 
{\sl contactomorphism} if $\phi^{\#}\theta$ is pointwise proportional to $\theta$. In other words, contactomorphisms preserve the \emph{contact structure} $\mathrm{ker}
(\theta)$.

\bigskip

The stratification of the Lie algebra $\mfrak h$ induces a family of anisotropic dilations
$\delta_\lambda$, $\lambda>0$ in $\he n$. The homogeneous dimension of $\he n$
with respect to $\delta_\lambda$, $\lambda>0$ equals $Q:=2n+2$. Unfortunately, when dealing 
with differential forms in $\he n$, 
the de Rham complex lacks scale invariance under anisotropic dilations. 
Thus, a substitute for de Rham's complex, that recovers scale invariance under $\delta_t$ has been defined 
by M. Rumin, \cite{rumin_jdg}. In turn, this notion makes sense for arbitrary contact manifolds. 
We refer to  \cite{rumin_jdg} and \cite{BFTT} for details of the construction. In the present paper, we shall merely need the following list of formal properties.
\medskip
\begin{itemize}
\medskip\item For $h=0,\ldots,2n+1$, the space of Rumin $h$-forms, $E_0^{h}$ is the space of smooth sections of a left-invariant subbundle of $\bigwedge^hT^*\he{n}$ (that we still denote by $E_0^{h}$). 
Hence it inherits inner products, $L^p$ and $W^{s,p}$ norms.
\medskip
\item A differential operator $d_c:E_0^{h}\to E_0^{h+1}$ is defined. It is left-invariant, 
homogeneous with respect to group dilations. It is a first order homogeneous operator
in the horizontal derivatives in degree $\neq n$, whereas {\sl it is a second
order homogeneous horizontal operator in degree $n$}. 
\medskip\item Contactomorphisms $\phi$ between open subsets of $\he{n}$ pull-back Rumin forms to Rumin forms, 
and in addition commute with $d_c$: 
$$
d_c(\phi^{\#}\alpha)=\phi^{\#}(d_c\alpha).
$$
\item The $L^2$ (formal) adjoint of $d_c$ is a differential operator $d_c^*$ of the same order as $d_c$.
\medskip\item Hypoelliptic ``Laplacians'' can be formed from $d_c$ and $d_c^*$ (see Definition \ref{rumin laplacian} below).
\medskip\item Altogether, operators $d_c$ form a complex: $d_c\circ d_c=0$.
\medskip\item This complex is homotopic to de Rham's complex $(\Omega^\bullet,d)$. The homotopy is achieved by differential operators $\Pi_E:E_0^\bullet\to\Omega^\bullet$ and $\Pi_{E_0}:\Omega^\bullet\to E_0^\bullet$ ($\Pi_E$ has horizontal order $\le 1$ and $\Pi_{E_0}$ is an algebraic operator).
\end{itemize}
In other words, $\Pi_E :E_0^\bullet\to \Omega^\bullet$ and $\Pi_{E_0}:\Omega^\bullet\to E_0^\bullet$ intertwine differentials $d_c$ and $d$,
$$
\begin{CD}
\cdots@>{d_c}>> E_0^h@>{d_c}>> E_0^{h+1} @>{d_c}>>\cdots
\\
 @.@V{\Pi_E}VV  @ V{\Pi_E}VV  @.\\
\cdots@>{d}>> \Omega^h  @>{d}>> \Omega^{h+1}@>{d}>>\cdots
\end{CD}
$$
$$
\begin{CD}
\cdots@>{d_c}>> E_0^h@>{d_c}>> E_0^{h+1} @>{d_c}>>\cdots
\\
 @.@AA{\Pi_{E_0}}A  @ AA{\Pi_{E_0}}A  @.\\
\cdots@>{d}>> \Omega^h  @>{d}>> \Omega^{h+1}@>{d}>>\cdots
\end{CD}
$$
and there exists an algebraic operator $A:\Omega^\bullet\to\Omega^{\bullet-1}$ such that $1-\Pi_{E_0}\Pi_E\Pi_E\Pi_{E_0}=0$ on $E_0^\bullet$ and $1-\Pi_E\Pi_{E_0}\Pi_{E_0}\Pi_E=dA+Ad$ on $\Omega^\bullet$. 

\subsection{Leibniz formula}

When $d_c$ is second order, $(E_0^\bullet,d_c)$ stops behaving like a differential module. This is the source of many complications.

\begin{lemma} [see also \cite{BFP2}, Lemma 3.2] \label{leibniz} If $\zeta$ is a smooth real function, then
\begin{itemize}
\item[i)] if $h\neq n$, then on $E_0^h$ we have:
$$
[d_c,\zeta] = P_0^h(W\zeta),
$$
where $P_0^h(W\zeta): E_0^h \to E_0^{h+1}$ is a linear homogeneous differential operator of order zero with coefficients depending
only on the horizontal derivatives of $\zeta$. If $h\neq n+1$, an analogous statement holds if we replace
$d_c$ by $d^*_c$;
\item[ii)] if $h= n$, then on $E_0^n$ we have
$$
[d_c,\zeta] = P_1^n(W\zeta) + P_0^n(W^2\zeta) ,
$$
where $P_1^n(W\zeta):E_0^n \to E_0^{n+1}$ is a linear homogeneous differential operator of order 1 (and therefore
horizontal) with coefficients depending
only on the horizontal derivatives of $\zeta$, and where $P_0^h(W^2\zeta): E_0^n \to E_0^{n+1}$ is a linear homogeneous differential operator in
the horizontal derivatives of order 0
 with coefficients depending
only on second order horizontal derivatives of $\zeta$. If $h = n+1$, an analogous statement holds if we replace
$d_c$ by $d^*_c$.
\item[iii)] if $h\neq n+1$, then
$$
[d_cd^*_c,\zeta] = P_1^h(W\zeta) + P_0^h(W^2\zeta) ,
$$
where $P_1^h(W\zeta):E_0^h \to E_0^{h}$ is a linear homogeneous differential operator of order 1 and therefore
horizontal) with coefficients depending
only on the horizontal derivatives of $\zeta$, and where $P_0^h(W^2\zeta): E_0^h \to E_0^{h}$ is a linear homogeneous differential operator in
the horizontal derivatives of order 0
 with coefficients depending
only on second order horizontal derivatives of $\zeta$.
\item[iv)] if $h= n+1$, then
$$
[d_cd^*_c,\zeta] = P_3^{n+1}(W\zeta) +P_2^{n+1}(W^2\zeta)  +P_1^{n+1}(W^3\zeta)  + P_0^{n+1}(W^4\zeta) ,
$$
where for $j=0,1,2,3$, the $P_j^{n+1}(W^{4-j}\zeta):E_0^{n+1} \to E_0^{n+1}$ are linear homogeneous differential operators of order j and therefore
horizontal) with coefficients depending
only on the horizontal derivatives of order $4-j$ of $\zeta$.

\end{itemize}

\end{lemma}

\begin{proof} The first two assertions are more or less straighforward. Let us prove the third
assertion. If $u$ is a Rumin's differential form of degree $h$, keeping in mind
the first assertion, we have
\begin{equation*}\begin{split}
[d_cd^*_c, \zeta]u & = d_cd^*_c (\zeta u) - \zeta d_cd^*_c u =
d_c(\zeta d^*_c u + P_0^h (W\zeta)u) -  \zeta d_cd^*_c u 
\\&
=\zeta d_cd^*_c u + P_0^{h-1}(W\zeta)d^*_c u + d_c(P_0^h (W\zeta)u) -  \zeta d_cd^*_c u 
\\&
= P_0^{h-1}(W\zeta)d^*_c u + d_c(P_0^h (W\zeta)u).
\end{split}\end{equation*}

Denote by $\{\xi_j^h\}$  a left-invariant basis of $E_0^h$. If $u=\sum_j u_j \xi_j^h$,
then $P_0^h (W\zeta)v= \sum_{j,k} a_{j,k} (W_k\zeta)u_k\xi_j^h$. Thus, using
 i) on $v_k\xi_j^h$ for all $j,k$, we get
\begin{equation*}\begin{split}
d_c(P_0^h (W\zeta)u & = \sum_{j,k} a_{j,k} (W_k\zeta) d_c(u_k\xi_j^h)
+ \sum_{j,k}P_0^h(W(W_k\zeta))u_k\xi_j^h
\\&
=: P_1^h(W\zeta)u + P_0^h(W^2\zeta)u.
\end{split}\end{equation*}
Therefore
$$
[d_cd^*_c, \zeta]u  = P_0^{h-1}(W\zeta)d^*_c u + P_1^h(W\zeta)u + P_0^h(W^2\zeta)u,
$$
and the assertion follows if we still denote by $P_1^h(W\zeta)$ the above operator
$P_0^{h-1}(W\zeta)d^*_c + P_1^h(W\zeta)$.

The proof of iv) is similar.
\end{proof}

\subsection{Rumin's Laplacian}

 \begin{definition}\label{rumin laplacian} 
In $\he n$, following \cite{rumin_jdg}, we define
the operators $\Delta_{\he{},h}$  on $E_0^h$ by setting
\begin{equation*}
\Delta_{\he{},h}=
\left\{
  \begin{array}{lcl}
     d_cd^*_c+d^*_c d_c\quad &\mbox{if } & h\neq n, n+1;
     \\ (d_cd^*_c)^2 +d^*_cd_c\quad& \mbox{if } & h=n;
     \\d_cd^*_c+(d^*_c d_c)^2 \quad &\mbox{if }  & h=n+1.
  \end{array}
\right.
\end{equation*}

\end{definition}

Notice that $-\Delta_{\he{},0} = \sum_{j=1}^{2n}(W_j^2)$ is the usual sub-Laplacian of
$\he n$. 

 For sake of simplicity, once a basis  of $E_0^h$
is fixed, the operator $\Delta_{\he{},h}$ can be identified with a matrix-valued map, still denoted
by $\Delta_{\he{},h}$
\begin{equation}\label{matrix form}
\Delta_{\he{},h} = (\Delta_{\he{},h}^{ij})_{i,j=1,\dots,N_h}: \mc D'(\he{n}, \rn{N_h})\to \mc D'(\he{n}, \rn{N_h}),
\end{equation}
where $\mc D'(\he{n}, \rn{N_h})$ is the space of vector-valued distributions on $\he n$, and $N_h$ is
the dimension of $E_0^h$ (see \cite{BBF}).

This identification makes possible to avoid the notion of currents: we refer to \cite{BFTT} for
a more elegant presentation.

\begin{definition} If a basis of $E_0^\bullet $ is fixed, and $1\le p \le \infty$,
we denote by $L^p(\he n, E_0^\bullet)$
the space of all sections of $E_0^\bullet$ such that their
components with respect to the given basis belong to
$L^p(\he n)$, endowed with its natural norm. Clearly, this definition
is independent of  the choice of the basis itself. 

The notations $M^p (\he n, E_0^\bullet)$ (see Definition \ref{M}), 
$\mc D(\he n, E_0^\bullet)$, $\mc S(\he n, E_0^\bullet)$, as well as
$W^{m,p}(\he{n}, E_0^\bullet)$ have the same meaning.
\end{definition}

It is proved in \cite{rumin_jdg} that $\Delta_{\he{},h}$ is
hypoelliptic and maximal hypoelliptic in the sense of \cite{HN}. { 
In
general, 
if $\mc L$ is a differential operator  on
$\mc D'(\he{n},\rn {N_h})$, then $\mc L$ is said hypoelliptic if 
for any open set $\mc V\subset \he{n}$ 
where $\mc L\alpha$ is smooth, then $\alpha$ is smooth in $\mc V$.
In addition, if $\mc L$ is
homogeneous of degree $a\in\mathbb N$,
we say that $\mc L$ is maximal hypoelliptic if
 for any  $\delta>0$ there exists $C=C(\delta)>0$ such that for any
homogeneous
polynomial $P$ in $W_1,\dots,W_{2n}$ of degree $a$
we have
$$
\|P\alpha\|_{L^{ 2}(\he n, \rn{N_h})}\le C
\left(
\|\mc L\alpha\|_{L^{ 2}(\he n, \rn{N_h})}+\|\alpha\|_{L^{ 2}(\he n, \rn{N_h})}
\right).
$$
for any $\alpha\in \mc D(B_\rho(0,\delta),\rn {N_h})$.}

Combining \cite{rumin_jdg}, Section 3,   and \cite{BFT3}, Theorems 3.1 and 4.1, we obtain the following result.

\begin{theorem}[see \cite{BFT3}, Theorem 3.1] \label{global solution}
If $0\le h\le 2n+1$, then the differential operator $\Delta_{\he{},h}$ is
homogeneous of degree $a$ with respect to group dilations, where $a=2$ if $h\neq n, n+1$ and  $a=4$ 
if $h=n, n+1$. It follows that
\begin{enumerate}
\item[i)] for $j=1,\dots,N_h$ there exists
\begin{equation}\label{numero}
    K_j =
\big(K_{1j},\dots, K_{N_h j}\big), \quad j=1,\dots N_h
\end{equation}
 with $K_{ij}\in\mc D'(\he{n})\cap \mc
E(\he{n} \setminus\{0\})$,
$i,j =1,\dots,N$;
\item[ii)] if $a<Q$, then the $K_{ij}$'s are
kernels of type $a$
 for
$i,j
=1,\dots, N_h$

 If $a=Q$,
then the $K_{ij}$'s satisfy the logarithmic estimate
$|K_{ij}(p)|\le C(1+|\ln\rho(p)|)$ and hence
belong to $L^1_{\mathrm{loc}}(\he{n})$.
Moreover, their horizontal derivatives  $W_\ell K_{ij}$,
$\ell=1,\dots,2n$, are
kernels of type $Q-1$;
\item[iii)] when $\alpha\in
\mc D(\he{n},\rn {N_h})$,
if we set
\begin{equation}\label{numero2}
   \Delta_{\he{},h}^{-1}\alpha:=
\big(    
    \sum_{j}\alpha_j\ast  K_{1j},\dots,
     \sum_{j}\alpha_j\ast  K_{N_hj}\big),
\end{equation}
 then $ \Delta_{h}\Delta_{\he{},h}^{-1}\alpha =  \alpha. $
Moreover, if $a<Q$, also $\Delta_{\he{},h}^{-1}\Delta_{h} \alpha =\alpha$.

\item[iv)] if $a=Q$, then for any $\alpha\in
\mc D(\he{n},\rn {N_h})$ there exists 
$\beta_\alpha:=(\beta_1,\dots,\beta_{N_h})\in \rn{N_h}$,  such that
$$\Delta_{\he{},h}^{-1}\Delta_{h}\alpha - \alpha = \beta_\alpha.$$

\end{enumerate}
\end{theorem}

{   \begin{remark}\label{pavidi}
If $a<Q$, 
 $ \Delta_{\he{},h} (\Delta_{\he{},h}^{-1} - \ccheck \Delta_{\he{},h}^{-1}) = 0$ and hence $\Delta_{\he{},h}^{-1} = \ccheck \Delta_{\he{},h}^{-1}$,
by the Liouville-type theorem of \cite{BFT3}, Proposition 3.2.
\end{remark}

\begin{remark} From now on, if there are no possible misunderstandings, we identify $\Delta_{\he{},h}^{-1}$ with its kernel.
\end{remark}
 }

 We notice that, if $n>1$, then $ \Delta_{\he{},h}^{-1}$ is associated with a kernel of type 2 or 4 and therefore 
 $ \Delta_{\he{},h}^{-1}f $ is well defined when  $f\in L^1(\he{n}, E_0^h)$. More precisely,  by 
Lemma \ref{convolutions} we have:

\begin{lemma} \label{convolution Delta} If $n>1$, $1\le h\le 2n$, then for any  horizontal
differential operator $W^I$ with homogeneous order $d(I)$, we have
\begin{itemize}
\item[i)] if $h\neq n, n+1$ and $d(I)=1$, then
$$
 \| f\ast W^I \Delta_{\he{}, h}^{-1}\|_{M^{Q/(Q-1)}} \le C\|f \|_{L^1({\he n, E_0^h})}
   $$
   for all $f\in L^1(\he n, E_0^h)$;
\item[ii)]   if $h=n, n+1$ and $1\le d(I)<4$, then
$$
 \| f\ast W^I \Delta_{\he{}, h}^{-1}\|_{M^{Q/(Q-4+d(I))}} \le C\|f \|_{L^1({\he n, E_0^h})}
   $$
   for all $f\in L^1(\he n, E_0^h)$.
\end{itemize}
By Corollary \ref{marc alternative coroll}, in both cases $ f\ast W^I \Delta_{\he{}, h}^{-1}\in L^1_{\mathrm{loc}}(\he n, E_0^*)$.
In particular, the map
\begin{equation}\label{L1-L1}
\Delta_{\he{}, h}^{-1}: L^1(\he n, E_0^*)\longrightarrow L^1_{\mathrm{loc}}(\he n, E_0^*)
\end{equation}
is continuous.
\end{lemma}

\begin{remark}\label{closed bis} Let $n> 1$. If $\alpha\in L^1(\he n, E_0^\bullet)$, then $\Delta^{-1}_{\mathbb H, h}\alpha$
 is well defined and belongs to 
$ L^1_{\mathrm{loc}}(\he n, E_0^\bullet)$. In particular, is a vector-valued distribution.
By Lemma \ref{closed ex bis},
if  
$\psi\in \mc D(\he n, E_0^h)$, then 
\begin{equation}
\Scal{\Delta^{-1}_{\mathbb H, h} \alpha}{ \psi} := \Scal{ \alpha}{ \Delta^{-1}_{\mathbb H, h}\psi}.
\end{equation}
\end{remark}
In this equation, the left hand side is the action of a matrix-valued distribution on a vector-valued test function, see formula \eqref{matrix form}, 
whereas the right hand side is (with a slight abuse of notation, since $ \Delta^{-1}_{\mathbb H, h}\psi$ is not a test function) 
the inner product of an $L^1$ vector-valued function with a $L^\infty$ vector-valued function.

\subsection{Commutation relations}

Typically, the operator used to invert $d_c$ is $d_c^*\Delta_{\mathbb H}^{-1}$. It inverts $d_c$ because $d_c$ commutes with $\Delta_{\mathbb H}^{-1}$.
Since $d_c$ and $d_c^*$ commute with $\Delta_{\mathbb H}$, it is natural that they commute with its inverse. 
One first shows this for test forms, and then (in a slightly weaker form) for $L^1$ forms by duality. 

\begin{lemma}[see \cite{BFP2}, Lemma 4.11]\label{comm} 
If $\alpha\in\mc D(\he n, E_0^h)$ and $n\ge 1$, then
\begin{itemize}
\item[i)]$
d_c \Delta^{-1}_{\mathbb H, h}\alpha = \Delta^{-1}_{\mathbb H, h+1} d_c\alpha$, \qquad $h=0,1,\dots, 2n$, 
\qquad $h\neq n-1, n+1$.

\item[ii)] $d_c \Delta^{-1}_{\mathbb H, n-1}\alpha = d_cd^*_c\Delta^{-1}_{\mathbb H, n} d_c\alpha$ \qquad ($h=n-1$).

\item[iii)]$
d_cd^*_c d_c \Delta^{-1}_{\mathbb H, n+1}\alpha = \Delta^{-1}_{\mathbb H, n+2} d_c\alpha$,
 \qquad ($h=n+1$).

\item[iv)]$d^*_c \Delta^{-1}_{\mathbb H, h}\alpha = \Delta^{-1}_{\mathbb H, h-1} d^*_c\alpha$
 \qquad $ h=1,\dots, 2n+1$, \qquad $h\neq n, n+2$.
 
 \item[v)] $d^*_c \Delta^{-1}_{\mathbb H, n+2}\alpha = d^*_c d_c\Delta^{-1}_{\mathbb H, n+1} 
 d^*_c\alpha$ \qquad ($h=n+2$).

\item[vi)]$
d^*_c d_c d^*_c  \Delta^{-1}_{\mathbb H, n}\alpha = \Delta^{-1}_{\mathbb H, n-1} d^*_c \alpha$,
 \qquad ($h=n$).
\end{itemize}
\end{lemma}

\begin{lemma}\label{commbis}
Let $h\geq 1$. Let  $\omega\in L^1(\he n, E_0^{h})$ be a $d_c$-closed form. Then $\Delta^{-1}_{\mathbb H, h}\omega$
is well defined and belongs to $L^1_{\mathrm{loc}}(\he n, E_0^h)$.
Furthermore, $d_c^*d_c \Delta^{-1}_{\mathbb H, h}\omega=0$ in distributional sense.
\end{lemma}

\begin{proof}
Let $\phi\in \mc D(\he n, E_0^h)$. By definition, by Corollary \ref{closed bis} and by Lemma \ref{comm}, iv), 
$$
\Scal{d_c^* d_c \Delta_{\he{},h}^{-1}\omega}{\phi}: = \Scal{d_c \Delta_{\he{},h}^{-1}\omega}{d_c\phi} = \Scal{\Delta_{\he{},h}^{-1}\omega}{d^*_cd_c\phi} = \Scal{\omega}{\Delta_{\he{},h}^{-1} d^*_cd_c\phi}
=\Scal{\omega}{d_c^* \Delta_{\he{},h}^{-1}d_c\phi}.
$$
One can write $ \Delta_{\he{},h}^{-1}d_c\phi=\phi\ast K$ where $K$ is a kernel of type $1$ or $2$. 
Let us show that
\begin{equation}\label{closed eq:1}
\int \scal{\omega}{ d_c^*(\phi \ast K)}\, dx=0.
\end{equation}
By Lemma \ref{closed}, $\phi \ast K$ is smooth and bounded on $\he n$, as well as all its horizontal derivatives. If $N\in \mathbb N$, let $\sigma_N$ be a cut-off function supported
 in $B(e,N+1)$ and identically 1 on $B(e,N)$. By dominated convergence theorem
 $$
 \int \scal{\omega}{ d_c^*(\phi \ast K)}\, dx  = \lim_{N\to\infty}\int \scal{\omega}{ \sigma_Nd_c^*(\phi \ast K)}\, dx.
 $$
 On the other hand, by Lemma \ref{leibniz},
 \begin{equation*}\begin{split}
&\int \scal{\omega}{ \sigma_Nd_c^*(\psi \ast K)}\, dx = \int \scal{\omega}{d_c^*( \phi \ast K)}\, dx
+  \int \scal{\omega}{ [d_c^*, \sigma_N](\phi \ast K)}\, dx
\\&\hphantom{xxxx}
=\Scal{d_c\omega}{ \sigma_N(\phi \ast K)}
+  \int \scal{\omega}{ [d_c^*, \sigma_N](\phi \ast K)}\, dx
\\&\hphantom{xxxx}
=
 \int \scal{\omega}{ [d_c^*, \sigma_N](\phi \ast K)}\, dx
\longrightarrow 0\qquad\mbox{as $N\to\infty$,}
\end{split}\end{equation*}
again by dominated convergence theorem, since horizontal derivatives of any order of $\sigma_N$ vanish as
$N\to\infty$ and $\phi \ast K\in L^\infty (\he n)$ by Lemma \ref{closed}. We conclude that $\Scal{d_c^* d_c \Delta_{\he{},h}^{-1}\omega}{\phi}=0$ for all test forms, hence $d_c^* d_c \Delta_{\he{},h}^{-1}\omega=0$.
 
\end{proof}

\section{Gagliardo-Nirenberg inequalities}\label{GNI}

The following is the core estimate of the paper. It provides primitives for globally defined 
$d_c$-closed $L^1$ forms, under an extra assumption on the vanishing of averages. This assumption 
is necessary. Indeed, it is obviously satisfied for forms admitting a compactly supported primitive. 
The extension to $L^1$ primitives is not hard, see Lemma \ref{by parts}. The case of forms 
admitting an $L^q$ primitive for some $q>1$ is more subtle, we refer to \cite{Tripaldi}. 

The starting point is the collection of Gagliardo-Nirenberg inequalities proven in \cite{BFP}. 

\begin{theorem}[\cite{BFP}, Theorem 1.6]\label{bfp1}
Let $u$ be a compactly supported Rumin $(h-1)$-form on $\he{n}$. Assume that $d_c^*u=0$. Then
\begin{align}
\label{GNBFP1}
\|u\|_{L^{Q/(Q-1)}(\he n,E_0^{h-1})}&\le C \|d_c u\|_{L^{1}(\he n,E_0^{h})}\quad \text{if }h\not=n+1,2n+1\\
\|u\|_{L^{Q/(Q-2)}(\he n,E_0^{n})}&\le C \|d_c u\|_{L^{1}(\he n,E_0^{n+1})}\quad \text{if }h=n+1.
\end{align}
\end{theorem}
In \cite{BFP}, Theorem 1.6, the first case corresponds to statements i), first line ($h=1$), ii) first line ($h=2$), iv) fourth line ($h=n+2$), iii) (other values of $h\not=2n+1$),
and the second case to statement iv) first line ($h=n+1$).

Given a $d_c$-closed $L^1$ form $\omega$, one would like to apply Theorem \ref{bfp1} to $u=d^*_c \Delta_{\he{},h}^{-1}\omega$, since $d_c^*u=0$. Since $u$ is not compactly supported, a cut-off is necessary, but this produces error terms which can be estimated thanks to Theorem \ref{anuli}, provided averages vanish.

\subsection{Estimate for $\he{n}$, $n>1$}

\begin{theorem}\label{core} Denote by $L^1_0(\he n,E_0^\bullet)$ the subspace of $L^1(\he n,E_0^\bullet)$
of forms with vanishing average. If $n>1$ we have:
\begin{itemize}
\item[i)] if $h\neq n, n+1$ and $1\leq h < 2n+1$, then
$$
\|d^*_c \Delta_{\he{},h}^{-1} \omega\|_{L^{Q/(Q-1)}(\he n, E_0^{h-1})}\le C \|\omega\|_{L^{1}(\he n, E_0^h)}\qquad\mbox{for all $\omega\in L^1_0(\he n,E_0^h)\cap \ker d_c$};
$$
\item[ii)] if $h=n$, then
$$
\|d^*_c d_cd^*_c \Delta_{\he{},n}^{-1} \omega\|_{L^{Q/(Q-1)}(\he n, E_0^{n-1})}\le C \|\omega\|_{L^{1}(\he n)}\qquad\mbox{for all 
$\omega\in L^1_0(\he n, E_0^n)\cap \ker d_c$};
$$
\item[iii)] if $h=n+1$, then
$$
\|d^*_c \Delta_{\he{},n+1}^{-1} \omega\|_{L^{Q/(Q-2)}(\he n, E_0^n)}\le C \|\omega\|_{L^{1}(\he n)}\qquad\mbox{for all $\omega\in L^1_0(\he n,E_0^{n+1})\cap \ker d_c$};
$$
\end{itemize}
If $n=1$, then statement iv) still holds with $h=3$. 
\end{theorem}
If $n=1$ and $h=1,2$, similar (but slightly different) inequalities are discussed in Proposition
\ref{n=1}.

\begin{proof}
We notice that, if $n>1$, then $ \Delta_{\he{},h}^{-1}$ is a kernel of type 2 or 4 and therefore, as we pointed out in \eqref{L1-L1}, if $\omega \in L^1_0(\he n,E_0^h)$, then $\Delta_{\he{},h}^{-1}\omega$ is well defined
and belongs to $L^1_{\mathrm{loc}}(\he n, E_0^h)$ for $1\le h\le 2n+1$. Thus we can consider the
convolution operator $\omega\mapsto d_c^*\Delta_{\he{},h}^{-1}\omega$ that is associated with a kernel of type 1
if $h\neq n, n+1$ and of type 2 if $h=n+1$. Analogously, if $h=n$, then the
convolution operator $\omega\to d^*_c d_cd^*_c \Delta_{\he{},h}^{-1} \omega$ is associated with a kernel of type 1.

 If $N\in\mathbb N$, let now $\chi_N$ be a cut-off function supported in $B(e,2N)$, $\chi_N\equiv 1$ on $B(e,N)$. If $\eps<1$ let
$J_\epsilon$ be an usual
Friedrichs' mollifier (for the group structure). Then, set
\begin{equation}\label{v eps,N h neq n}
v_{\eps,N} :=J_\eps\ast  d^*_c (\chi_N \Delta_{\he{},h}^{-1} \omega)\qquad\mbox{if $h\neq n$,}
\end{equation}
while
\begin{equation}\label{v eps,N h = n}
v_{\eps,N} :=J_\eps\ast d^*_c (\chi_N d_c d^*_c  \Delta_{\he{},h}^{-1} \omega)\qquad\mbox{if $h=n$.}
\end{equation}
Notice now that both $d^*_c (\chi_N \Delta_{\he{},h}^{-1} \omega)$ if $h\neq n$, and
$d^*_c (\chi_N d_c d^*_c  \Delta_{\he{},h}^{-1} \omega) $ if $h= n$ are compactly supported and uniformly bounded in $L^1(\he n,E_0^{h-1})$.
Indeed, in the first case we can write
$$
d^*_c (\chi_n \Delta_{\he{},h}^{-1} \omega) =  \chi_N(d^*_c \Delta_{\he{},h}^{-1} \omega) + [d^*_c, \chi_N] \Delta_{\he{},h}^{-1} \omega,
$$
and both terms on the right hand side are bounded in $L^1(\he n,E_0^{h-1})$ by Lemmata \ref{leibniz}
and \ref{convolution Delta}. An analogous argument can be carried out in case $h=n$,
keeping in mind that $d^*_c$ is an operator of order 1 and $d_c d^*_c  \Delta_{\he{},h}^{-1}$ is associated with a kernel
of type 2.

We observe that
$$
d_c^*v_{\eps,N}=   J_\eps \ast (d^*_c)^2(\chi_N \Delta_{\he{},h}^{-1} \omega)=0,
$$
if $h\not=n$. In case $h=n$,
$$
d_c^*v_{\eps,N}=   J_\eps \ast (d^*_c)^2(\chi_N d_c d_c^*\Delta_{\he{},h}^{-1} \omega)=0
$$
again.

Let us prove sentences i) and iii) simultaneously. To avoid cumbersome notations, in the sequel
when  $L^p$-spaces are involved, we shall drop the target spaces. We apply Theorem \ref{bfp1} to $v_{\eps,N}$.
\begin{equation}\label{july 7 eq:1}\begin{split}
\|v_{\eps,N}&\|_{L^{Q/(Q-1)}(\he n, E_0^{h-1})} 
\le C  \| d_c v_{\eps,N}\|_{L^1(\he n)}
\\&
=
C \| J_\eps \ast d_c  d^*_c( \chi_N\Delta_{\he{},h}^{-1} \omega)\|_{L^1(\he n)} 
\\&
\le
C \big\{ \| J_\eps \ast [d_c d^*_c, \chi_N] (\Delta_{\he{},h}^{-1} \omega)\|_{L^1(\he n)} +
 \| J_\eps \ast  \chi_N (d_cd^*_c \Delta_{\he{},h}^{-1} \omega)\|_{L^1(\he n)}  \big\}
 \\&
\le
C \big\{ \| [d_cd^*_c, \chi_N] ( \Delta_{\he{},h}^{-1} \omega)\|_{L^1(\he n)} +
 \|  \chi_N (d_cd^*_c \Delta_{\he{},h}^{-1} \omega)\|_{L^1(\he n)}  \big\}.
  \end{split}\end{equation}
If $h=n+1$, then \eqref{july 7 eq:1} still holds
provided we replace $\|v_{\eps,N}\|_{L^{Q/(Q-1)}(\he n, E_0^{n})} $ with $\|v_{\eps,N}\|_{L^{Q/(Q-2)}(\he n, E_0^{n})} $.

By Lemma \ref{commbis}, 
\begin{equation}\label{jan 6 eq:2}
d_cd^*_c\Delta_{\he{},h}^{-1}\omega = \omega - d^*_c d_c \Delta_{\he{},h}^{-1}\omega = \omega.
\end{equation}

By Lemma \ref{leibniz}, $[d_cd^*_c, \chi_N]$ can be written as a sum of terms of the form $P_j^h(W^k)$ with $j=0,1$, $j+k=2$ if $h\neq n+1$,
and with $j=0,1,2,3$, $j+k=4$ if $h = n+1$. By Proposition \ref{kernel}, ii), in both cases the norm  $\| [d_cd^*_c, \chi_N] ( \Delta_{\he{},h}^{-1} \omega)\|_{L^1(\he n)}$ can be estimated
by a sum of terms of the form
$$
\frac{1}{N^k} \int_{B(e,2N)\setminus B(e,N)} \big|K\omega \big| \, dx,
$$
where $K$ is a kernel of type $k\ge 1$. Thus, we can apply Theorem \ref{anuli} to conclude that
$$
 \| [d_cd^*_c, \chi_N] ( \Delta_{\he{},h}^{-1} \omega)\|_{L^1(\he n)}\longrightarrow 0\qquad \mbox{as $N\to \infty$.}
$$
If $\eps\to 0$, then $v_{\eps,N} \to d^*_c  (\chi_N \Delta_{\he{},h}^{-1} \omega)$ in $L^1(\he n, E_0^{h-1})$ 
(and therefore we may assume a.e.). By Fatou's theorem, this provides an $L^{Q/(Q-1)}$ bound on $d^*_c  (\chi_N \Delta_{\he{},h}^{-1} \omega)$. Since 
$$
d^*_c (\chi_N \Delta_{\he{},h}^{-1} \omega) =  \chi_N(d^*_c \Delta_{\he{},h}^{-1} \omega) + [d^*_c, \chi_N] \Delta_{\he{},h}^{-1} \omega,
$$
as $N$ tends to $\infty$, this converges a.e. to $d^*_c \Delta_{\he{},h}^{-1} \omega$. By Fatou again, 
$$
\|d^*_c \Delta_{\he{},h}^{-1} \omega\|_{L^{Q/(Q-1)}(\he n)}\le C \|\omega\|_{L^{1}(\he n)}.
$$

This completes the proof of i) and iii). Finally, the proof of ii) can be carried out through the same argument,
provided we keep in mind
 Lemma \ref{comm}-i) in order to obtain that $d_cd^*_c d_c d^*_c \Delta_{\he{},h}^{-1}\omega = \omega$. 
\end{proof}

\subsection{The case of $\he1$}

When $n=1$ and $h=1$ or $2$, the statement and the proof are slightly different,
due to the fact that $\Delta_{\he{},h}^{-1}$ in degrees $h=1,2$ has a logarithmic behavior, since the
order of $\Delta_{\he{},h}$ equals the homogeneous dimension $Q=4$ (see Theorem \ref{global solution}, ii)). Incidentally, if $h=0$ or $h=3$ the order of $\Delta_{\he{},h}$ is $2<Q$ and there is no difference from the case $n>1$.

If $h=1,2$, the way to circumvent this obstacle is to avoid mentioning $\Delta_{\he{},h}^{-1}$ and focus on $d_c^*\Delta_{\he{},h}^{-1}$ which is still given by convolution with a kernel of type $2$ or $3$. Indeed,
suppose first $\omega\in L^1(\he 1,E_0^h)$, is compactly supported. Then, keeping in 
mind that $K\in L^1_{\mathrm{loc}}$, by  
Theorem \ref{global solution}, ii) again, we obtain that 
$\Delta_{\he{},h}^{-1} \omega \in L^1_{\mathrm{loc}}(\he 1,E_0^h)$ (therefore it is a 
distribution), and we can write
\begin{equation}\label{kappa tilde}
d_c^*\Delta_{\he{},h}^{-1} \omega=: \omega\ast \tilde K,
\end{equation}
where (keeping in mind Proposition \ref{kernel})
$\tilde K$ is a kernel of type $3$ if $h=1$ and of type $2$ if $h=2$. 

\begin{proposition}\label{n=1}    
Assume that $h=1$ or $2$ and $n=1$. Let $\tilde K$ the convolution operator 
associated with the kernel $\tilde K$. Let
$\omega\in L^1(\he 1,E_0^h)$ without any support assumption, then $\tilde K \omega\in L^1_{\mathrm{loc}}(\he 1,E_0^h)$, 
\begin{equation}\label{n=1 eq:1}
\|d^*_c d_c \tilde K \omega\|_{L^{Q/(Q-1)}(\he 1,E_0^0)}\le C \|\omega\|_{L^{1}(\he 1, E_0^1)}\qquad\mbox{for all $\omega\in L^1_0(\he 1, E_0^1)\cap \ker d_c$};
\end{equation}
and
\begin{equation}\label{n=1 eq:2}
\| \tilde K \omega\|_{L^{Q/(Q-2)}(\he 1, E_0^1)}\le C \|\omega\|_{L^{1}(\he n, E_0^2)}\qquad\mbox{for all $\omega\in L^1_0(\he 1, E_0^2)\cap \ker d_c$};
\end{equation}
\end{proposition}

\begin{proof}
Let us prove for instance \eqref{n=1 eq:2}. The proof is a mere reformulation of that of Theorem \ref{core}.

Take  $\omega\in L^1_0(\he 1, E_0^2)\cap \ker d_c$. First of all,
we want to show that 
\begin{equation}\label{dc dc* tilde K} 
d_c\tilde K\omega =\omega \qquad\mbox{and}\qquad d_c^*K\omega =0.
\end{equation}

To this end, we take a sequence $(\omega_N)_{N\in\mathbb N}$
of compactly supported forms converging to $\omega $ in $L^1(\he 1, E_0^2)$. It is easy to see that 
$$
d_c^* \Delta_2^{-1}\omega_N =\tilde K \omega_N \to \tilde K \omega\qquad\mbox{in $ L^1_{\mathrm{loc}}(\he 1,E_0^1)$ as $N\to\infty$,}
$$
and hence $\tilde K \omega_N \to \tilde K \omega$ in $ \mc D'(\he 2,E_0^1)$ together with all their derivatives. In particular
\begin{equation}\label{dc* tilde K} d_c^*\tilde K\omega = \lim_{N\to\infty} d_c^* d_c^* \Delta_2^{-1}\omega_N=0.
\end{equation}

By Lemma \ref{commbis}
(and keeping in mind Remark \ref{closed ex bis remark}) 
for all $\phi\in \mc D(\he 2,E_0^3)$, so that
\begin{equation*}\begin{split}
&\lim_{N\to\infty} \Scal{d_c^*d_c d_c^* d_c \Delta_2^{-1}\omega_N}{\phi} : = \lim_{N\to\infty} \Scal{ \Delta_2^{-1}\omega_N}{ d_c^* d_c d_c^*d_c\phi}
\\&\hphantom{xxxx}
= \lim_{N\to\infty} \Scal{ \omega_N}{ \Delta_2^{-1} d_c^* d_c d_c^*d_c\phi} = \lim_{N\to\infty}  \Scal{ \omega_N}{ d_c^* \Delta_2^{-1} d_c d_c^*d_c\phi}.
\end{split}\end{equation*}
On the other hand, since $d_c^* \Delta_2^{-1}$ is a kernel of type 2 and hence $d_c^* \Delta_2^{-1} d_c d_c^*d_c\phi
\in L^{\infty} (\he 2,E_0^3)$,
\begin{equation*}\begin{split}
\lim_{N\to\infty}  \Scal{ \omega_N}{ d_c^* \Delta_2^{-1} d_c d_c^*d_c\phi}
=  \Scal{ u}{ d_c^* \Delta_2^{-1} d_c d_c^*d_c\phi} = 0
\end{split}\end{equation*}
i.e.
\begin{equation}\label{june 19 eq:1}\begin{split}
\lim_{N\to\infty} \Scal{d_c^*d_c d_c^* d_c\Delta_{\he{},h}^{-1}\omega_N}{\phi} = 0.
\end{split}\end{equation}
Therefore, by Theorem \ref{global solution}, iv) there exists a left invariant form
$\beta = \beta(\phi)$ such that
\begin{equation*}\begin{split}
\Scal{d_c\tilde Ku}{\phi} & = 
\lim_{N\to\infty }\Scal{d_c\tilde K\omega_N}{\phi} = \lim_{N\to\infty }\Scal{d_cd_c^*\Delta_2^{-1}\omega_N}{\phi} 
\\&
= \lim_{N\to\infty }\Scal{(d_cd_c^* + (d_c^*d_c)^2)\Delta_2^{-1}\omega_N}{\phi} 
= \lim_{N\to\infty }\Scal{\Delta_2\Delta_2^{-1}\omega_N}{\phi} 
\\&
= \lim_{N\to\infty }\Scal{\omega_N}{\Delta_2^{-1}\Delta_2\phi} 
\qquad\mbox{(by Remark \ref{closed ex bis remark})}
\\&
= \lim_{N\to\infty }\Scal{\omega_N}{\phi + \beta} 
= \Scal{\omega}{\phi + \beta}= \Scal{\omega}{\phi},
\end{split}\end{equation*}
i.e. $d_c\tilde K\omega =\omega$. Thus, the proof can be completed arguing basically as in the proof of Theorem \ref{core}. More
precisely,  let $\chi_N$ be a cut-off function supported in $B(e,2N)$, $\chi_N\equiv 1$ on $B(e,N)$. If $J_\epsilon$ is an usual
Friedrich's mollifier for $\eps<1$, let us consider
$$
v_{\eps,N} :=J_\eps\ast \chi_N (\tilde K \omega)
$$
(notice the slight difference from \eqref{v eps,N h neq n}, due to the fact that we cannot split $\tilde K$ as $d_c^* \Delta_2^{-1}$).
As in Theorem \ref{core}, $v_{\eps,N} \in \mc D(\he 1, E_0^{1})$ and $v_{\eps,N} \to \chi_N (\tilde K u)$ in $L^1(\he 1, E_0^{1})$ 
(and therefore we may assume a.e.) as $\eps\to 0$.  If we apply the estimates of \cite{BFP}, Theorem 1.3 - iv)
and \eqref{dc dc* tilde K} above, we get (to avoid cumbersome notations, when  $L^p$-spaces are involved, 
we shall drop the target spaces):
\begin{equation*}\begin{split}
\|v_{\eps,N}&\|_{L^{Q/(Q-2)}(\he 1)} 
\le C \big\{ \| d_c v_{\eps,N}\|_{L^1(\he 1)} + \| d_c d^*_c v_{\eps,N}\|_{L^1(\he 1)}\big\}
\\&
=
C \big\{ \| J_\eps \ast d_c \chi_N (\tilde K \omega)\|_{L^1(\he 1)} + \| J_\eps \ast d_c d^*_c(\chi_N (\tilde K \omega))\|_{L^1(\he 1)}\big\}
\\&
\le C \big\{ \|  d_c \chi_N (\tilde K \omega)\|_{L^1(\he 1)} + \| d_c d^*_c(\chi_N (\tilde K \omega))\|_{L^1(\he 1)}\big\}
\\&
\le
C \big\{   \|   \chi_N \omega\|_{L^1(\he 1)}+     \|  [d_c, \chi_N] (\tilde K \omega)\|_{L^1(\he 1)} 
 \\&
 \hphantom{\le}+
  \|  [d_c d^*_c,\chi_N] (\tilde K \omega)\|_{L^1(\he 1)}\big\}
\\&  \le
C \big\{   \|   u\|_{L^1(\he 1)}+     \|  [d_c, \chi_N] (\tilde K \omega)\|_{L^1(\he 1)} 
+  \|  [d_c d^*_c,\chi_N] (\tilde K \omega)\|_{L^1(\he 1)}\big\}.
   \end{split}\end{equation*}
   Thus, by Fatou's lemma
   \begin{equation*}\begin{split}
      \| \chi_N & (\tilde K \omega)\|_{L^1(\he n)} 
      \le C \big\{   \|   \omega\|_{L^1(\he 1)}
      \\&
      +     \|  [d_c, \chi_N] (\tilde K \omega)\|_{L^1(\he 1)} 
+  \|  [d_c d^*_c,\chi_N] (\tilde K\omega)\|_{L^1(\he 1)}\big\}.
   \end{split}\end{equation*}
 
   Keep now in mind that  $ \tilde K\omega$ is a form of degree 1, so that  
   both $d_c$ and $d_cd_c^*$ are horizontal operators of order 2.
   By Lemma \ref{leibniz} the two terms containing the commutators can be bounded by terms
   of the form
  $$
  \frac {1}{N^2} \int_{N\le |p|\le 2N} \big| \tilde K\omega\big|\, dp
  $$ 
   or by a sum of terms of the form
  $$
  \frac {1}{N} \int_{N\le |p|\le 2N} \big|W_\ell \tilde K\omega\big|\, dp.
  $$
  Thus we can conclude obtaining \eqref{n=1 eq:1} again by Fatou's lemma and Theorem \ref{anuli}.

\end{proof}
Once Theorem \ref{core} and Proposition \ref{n=1} are proved, the proof of Theorem \ref{poincareglobal} is straightforward:

\subsection{Proof of Theorem \ref{poincareglobal}}

In the Heisenberg case, Theorem \ref{core} and Proposition \ref{n=1} provide $L^q$ primitives (with the announced values of $q$) for $d_c$-closed $L^1$ forms with vanishing averages, in all degrees but the top degree. The Euclidean case is even simpler. This proves Theorem \ref{poincareglobal}.

\section{Interior inequalities}\label{InIn}

Interior inequalities are proven in three steps. Applying cut-offs on forms and on kernels, one first constructs a homotopy $K$ which slightly increases differentiability. Then Rumin's homotopy is used to replace Rumin forms with usual differential forms. Finally, Iwaniec-Lutoborsky's Euclidean homotopy is applied.

\subsection{The function space $L^1\cap d_c^{-1}L^1$}

A homotopy is an operator $K$ such that $d_cK+Kd_c$ equals identity (up to a loss on domain). To make sense of such an identity, one must restrict to forms $\alpha$ which belong to $L^1$ and such that $d_c\alpha$ belongs to $L^1$ as well.

 \begin{lemma}\label{density}
Let $B$ be a ball in $\he n$. We set
$$
(L^1\cap d_c^{-1}L^1)(B,E_0^\bullet):=\{\alpha\in L^1(B,E_0^\bullet)\,;\,d_c\alpha\in L^1(B,E_0^{\bullet+1})\},
$$
endowed with the graph norm. Then $C^\infty(B,E_0^\bullet)$ is dense in $(L^1\cap d_c^{-1}L^1)(B,E_0^\bullet)$.

\end{lemma}

\begin{proof} Take $u\in (L^1\cap d_c^{-1}L^1)(B,E_0^\bullet)$. If $u$ is compactly supported, then
it can be approximated by convolution with  Friedrichs' mollifiers $J_\eps$ for the structure of the group, since
$d_c(J_\eps\ast u) = J_\eps \ast d_cu$. The proof of the statement for non-compactly supported forms
can be carried out by mimicking verbatim the classical Meyers-Serrin's proof (see \cite{schwarz}, Theorem 1.3.3,
and \cite{FSSC_houston}).

\end{proof}

\begin{lemma}\label{L1 boundedness of Kd new}
Let $B$ be a ball in $\he n$. Set $K=d_c^*\Delta_{\mathbb H}^{-1}$  if $n >1$ and,
if $n=1$ is defined by \eqref{kappa tilde}. If
\begin{equation}\label{K0 formula}
K_0:= K\quad\text{in degree }h\neq n\text{ and}\qquad K_0:= d_c^*d_c K\quad\text{in degree }h=n.
\end{equation}
Then:
\begin{itemize}
\item $K_0$ is a kernel of type $1$ on forms of degree $h$, $h\neq n+1$ and of type 2 if $h=n+1$;

\item if $\chi$ is a smooth function with compact support in $B$, then the identity
$$
\chi=d_c K_0\chi+K_0d_c\chi
$$
holds on the space $(L^1\cap d_c^{-1}L^1)(B,E_0^\bullet)$.
\end{itemize}

\end{lemma}

\begin{proof}
If  $h\neq n-1, n, n+1$ and $\mathcal{D}(\he n,E_0^\bullet)$, then, by Theorem \ref{global solution}
and Lemma \ref{comm}, i),
\begin{equation*}\begin{split}
u & = d_c d_c^* \Delta_{\mathbb H}^{-1}u + d_c^* d_c \Delta_{\mathbb H}^{-1}u = d_c d_c^* \Delta_{\mathbb H}^{-1}u + d_c^*  \Delta_{\mathbb H}^{-1}d_cu
\\&
= d_c Ku + Kd_c u.
\end{split}\end{equation*}
If $h=n-1$, then 
\begin{equation*}\begin{split}
u & = d_c d_c^* \Delta_{\mathbb H}^{-1}u + d_c^* d_c \Delta_{\mathbb H}^{-1}u = d_c d_c^* \Delta_{\mathbb H}^{-1}u + d_c^*d_c d_c^*  \Delta_{\mathbb H}^{-1}d_cu
\\&
= d_c Ku + d_c^*d_c Kd_c u.
\end{split}\end{equation*}
If $h=n$, then 
\begin{equation*}\begin{split}
u & = (d_c d_c^*)^2 \Delta_{\mathbb H}^{-1}u + d_c^* d_c \Delta_{\mathbb H}^{-1}u = (d_c d_c^*)^2 \Delta_{\mathbb H}^{-1}u + d_c^*  \Delta_{\mathbb H}^{-1}d_cu
\\&
= d_c d_c^* d_c Ku + d_c^* Kd_c u.
\end{split}\end{equation*}
Finally, if $h=n+1$, then 
\begin{equation*}\begin{split}
u & = d_c d_c^* \Delta_{\mathbb H}^{-1}u + d_c^* d_c d_c^* d_c \Delta_{\mathbb H}^{-1}u = 
d_c d_c^* \Delta_{\mathbb H}^{-1}u + d_c^*  \Delta_{\mathbb H}^{-1}d_cu
\\&
= d_c  Ku + Kd_c u.
\end{split}\end{equation*}
In other words, with notations of \eqref{K0 formula}, for any $h$ we can write
$$
u = d_cK_0u + K_0d_c u,
$$
where $K_0$ is a kernel of type $1$ when it acts on forms of degree $h$, $h\neq n+1$ and of type 2 if $h=n+1$,

Take now $u\in (L^1\cap d_c^{-1}L^1)(B,E_0^h)$, $0\le h\le 2n+1$. By Lemma \ref{density} the exists a sequence $(u_N)_{N\in\N}$
of smooth $h$-forms on $B$ such that
$$
u_N \longrightarrow u  \qquad \mbox{in $L^1(\he n,E_0^h)$ as $N\to\infty$.}
$$
and
$$
d_cu_N  \longrightarrow d_cu\qquad \mbox{in $L^1(\he n,E_0^{h+1})$ as $N\to\infty$.}
$$
Obviously
$$
\chi u_N = d_cK_0 (\chi u_N) + K_0d_c(\chi u_N)\qquad\mbox{for all $N\in\N$.}
$$
Since $\chi u_N \to \chi u$ in $L^1(\he n,E_0^h)$ as $N\to \infty$, then
$K_0(\chi u_N) \to K_0(\chi u)$ in $L^1(B,E_0^{h-1})$, and $d_c K_0(\chi u_N) \to d_c K_0(\chi u)$ in the sense of distributions.
Let us consider now $K_0d_c(\chi u_N)= K_0(\chi d_cu_N + [d_c,\chi] u)$. Obviously, $\chi d_cu_N \to \chi d_cu$
in $L^1(\he n,E_0^{h+1})$ as $N\to \infty$, and then
$K_0(\chi d_cu_N) \to K_0(\chi d_cu)$ in $L^1(B,E_0^h)$.

Let us consider the term $K_0 [d_c,\chi] u_N$. If $h\neq n$, then, by Lemma \ref{leibniz},  $[d_c,\chi]u_N \to [d_c,\chi]u$
in $L^1(\he n,E_0^h)$ and we can conclude as above. Thus we are left with the case $h=n$. By Lemma \ref{leibniz}, 
$[d_c,\chi]u$ can be written as a sum of terms of the form $(W_iW_j \chi) u_N$ and of the form $W_j \{(W_i \chi)(u_N)_\ell\}$,
where $(u_N)_\ell$ is the $\ell$-th component of $u_N$. The terms of the form $(W_iW_j \chi) u_N$ can be handled as above.
On the other hand, $K_0W_j$ is a kernel of type 1, and, again, $(W_i \chi)(u_N)_\ell  \to (W_i \chi)u_\ell$ in $L^1(\he n)$
and we can conclude as above.
\end{proof}

\subsection{A local smoothing homotopy}

It is obtained by cutting off the global inverse of $d_c$ provided in Section \ref{GNI}. This operator can be applied only to global forms whose averages vanish. Therefore we begin by checking that averages vanish for $d_c$-exact forms.

  \begin{lemma}[see \cite{BFTT}, Remark 2.16]\label{by parts}
Let
  $\psi\in  L^1(\he n,E_0^{h})$ be a compactly supported form with $d_c\psi \in  L^1(\he n,E_0^{h+1})$,
  and let $\xi\in \bigwedge^{2n-h}$ be a left-invariant invariant form. Then
  $$
  \int_{\he{n} }  d_c\psi\wedge \xi = 0.
  $$
\end{lemma}

\begin{proof} By \cite{BFTT}, identity (16), we have
  \begin{equation*}
d_c\psi\wedge\xi = d_c\psi\wedge (\Pi_{E_0}\xi),
 \end{equation*}
 so that we can assume that $\xi\in E_0^{2n-h}$ (and $\xi$ is still a
 ``constant coefficient form''). Moreover, by Lemma \ref{density},
 we can assume that $\psi\in  \mc D(\he n,E_0^{h})$. Thus we can conclude
 by Remark 2.16 in \cite{BFTT}.
\end{proof}

\begin{proposition}\label{smoothing}
Let $B\Subset B'$ be concentric balls in $\he n$. For $h=1,\ldots,2n$, let $q=Q/(Q-1)$ if $h\not=n+1$ and $q=Q/(Q-2)$ if $h=n+1$. For every $s\in\N$, there exists a smoothing operator $S:L^1(B',E_0^h)\to W^{s,q}(B,E_0^{h-1})$ and a bounded operator $T:L^1(B',E_0^h)\cap d_c^{-1}(L^1(B',E_0^{h+1}))$ to $L^q(B,E_0^{h-1})$, and such that, for $L^1$-forms $\alpha$ on $B'$ such that $d_c\alpha\in L^1$, 
$$
\alpha=d_cT\alpha+T d_c \alpha+S\alpha\qquad \text{on }B.
$$
In particular, $d_c S=Sd_c$ on $L^1 \cap d_c^{-1}L^1$. Furthermore, there exist $r>0$ and $p>1$ such that for all $s\geq 0$, $T$ extends to a bounded operator $L^1(B',E_0^h)\to W^{r,p}(B,E_0^{h-1})$ and $W^{s,p}(B',E_0^h)\to W^{s+1,p}(B,E_0^{h-1})$. In degree $n+1$, if $W$ is a horizontal derivative, $W T$ extends to a bounded operator $L^1(B',E_0^n)\to W^{r,p}(B,E_0^{n})$ and $W^{s,p}(B',E_0^n)\to W^{s+1,p}(B,E_0^{n})$.

Finally, $T$ and $S$ merely enlarge by a small amount the support of differential forms.
\end{proposition}

 \begin{proof}
{  Le us fix two balls $B_0$, $B_1$ with 
\begin{equation}\label{varie B}
B\Subset B_0 \Subset B_1\Subset B',
\end{equation}
and a cut-off function $\chi \in \mc D(B_1)$,
 $\chi\equiv 1$ on  $B_0$. If $\alpha \in ( L^1 \cap d_c^{-1})(B', E_0^\bullet)$, we set $\alpha_0=  \chi\alpha$, continued by zero outside $B_1$.
 Denote by $k_0$  the kernel associated with $K_0$  in Lemma \ref{L1 boundedness of Kd new}. 
 We consider a cut-off function $\psi_R$ supported in a $R$-neighborhood
of the origin, such that $\psi_R\equiv 1$ near the origin. Then we can write  
$k_0=k_0\psi_R + (1-\psi_R)k_0$ Let us denote by $K_{0,R}$
 the convolution operator associated with $\psi_R k_0$.
By Lemma \ref{L1 boundedness of Kd new},
\begin{equation}\begin{split}\label{apr 20 eq:1}
\alpha_0 &=d_c K_0\alpha_0+K_0 d_c \alpha_0\\
&= d_c K_{0,R} \alpha_0 + K_{0,R}d_c \alpha_0 + S_0\alpha_0,
\end{split}\end{equation}
where $S_0$ is defined by
$$
S_0\alpha_0 :=  d_c( (1-\psi_R)k_0\ast \alpha_0) + (1-\psi_R)k_0 \ast d_c\alpha_0.
$$ 
 We set
$$
T_1\alpha := K_{0,R}\alpha_0, \qquad S_1\alpha:=  S_0\alpha_0.
$$
Since the kernel $\psi_R k_0\in L^1$, $K_{0,R}$ maps $L^1$ to $L^1$.

If $\beta\in L^1(B_1)$, we set
$$
T_1 \beta:= K_{0,R}(\chi\beta)_{\big|_{B}}, \qquad S_1\alpha:=  {S_0\alpha_0}_{\big|_{B}}.
$$
We notice that, provided $R>0$ is small enough, the values of $T_1\beta$  do not depend on the continuation of $\beta$ outside $B_1$.
Moreover
$$
{K_{0,R} d_c\alpha_0}_{\big|_B} = K_{0,R} d_c (\chi \alpha)_{\big|_B} = K_{0,R} (\chi d_c\alpha)_{\big|_B} =T_1(d_c\alpha),
$$
since $ d_c (\chi \alpha)\equiv \chi d_c\alpha$ on $B_0$. 
Thus, by \eqref{apr 20 eq:1},
 $$
\alpha = d_c T_1\alpha + T_1d_c \alpha + S_1\alpha \qquad\text{in }B.
$$

Write $\phi=T_1\alpha\in L^1(B_0)$. By difference, $d_c\phi=\alpha-S_1\alpha -T_1 d_c \alpha\in L^1(B_0)$.

 }

Unfortunately, so far one cannot assert that $\phi\in L^q(B_0)$ {  and we must
in some sense ``iterate'' the argument. Let us sketch how this iteration will work:
let $\zeta$ be a cut-off function supported in $B_0$, identically equal to $1$ in a neighborhood $\mc U$ of $B$, and set $\omega=d_c(\zeta\phi)$. 
Obviously, the form $\zeta\phi$ (and therefore also $\omega$) are defined on all $\he n$ and are compactly
supported in $B_0$. In addition, $\omega$ is closed.
Suppose for a while we are able to prove that
\begin{itemize}
\item[a)]  $\omega \in L^1(\he n)$;
\item[b)]  
$\| K_{0,R}\omega\|_{L^q(\he n)} \le C\|\alpha\|_{L^1(B')}$ for some $q>1$,
\end{itemize}
and let us show how the argument can be carried out (here and in the sequel of the proof,
to avoid cumbersome notations, 
when  $L^p$-spaces are involved, we  drop the target spaces).

First we stress that, if $R$ is small enough, then when $x\in B$, $K_{0,R}\omega(x)$
depends only on the restriction of $d_c\phi$ to $\mc U$, so that the map
$$
\alpha \to K_{0,R}\omega \big|_{B}
$$
is linear.

Notice that $\omega=\chi\omega$, so that, by \eqref{apr 20 eq:1},
$$
d_c(\zeta\phi) = \omega = d_cK_{0,R}\omega + S_0\omega.
$$
Therefore in $B$
$$
\alpha-S_1\alpha -T_1 d_c\alpha= d_c\phi = d_c(\zeta\phi) = d_cK_{0,R}\omega + S_0\omega,
$$
and then in $B$
\begin{equation*}\begin{split}
\alpha &= T_1 d_c\alpha + d_c (K_{0,R}\omega \big|_{B}) + S_1\alpha \big|_{B}+ S_0\omega \big|_{B}
\\&=T_1 d_c\alpha + d_c(K_{0,R}(\chi\omega) \big|_{B}) + S\alpha \\
&=:\bar T d_c\alpha + d_cT\alpha + S\alpha.
\end{split}\end{equation*}

First notice that the map $\alpha\to \omega=\omega(\alpha)$ is linear, and hence $\bar T$, $T$ and
$S$ are linear maps. In addition, by b),
\begin{equation*}\begin{split}
\| T\alpha\|_{L^q(B)} \le  \| K_{0,R}(\chi\omega)\|_{L^q(\he n)} =  \| K_{0,R}(\omega)\|_{L^q(\he n)}
\le C\,( \|\alpha\|_{L^1(B')}+ \|d\alpha\|_{L^1(B')}).
\end{split}\end{equation*}

As for the map $\alpha \to S\alpha$ we have just to point out that, when $x\in B$, $S\alpha(x)$
can be written as the convolution of $\alpha_0$ with a smooth kernel with bounded derivatives
of any order.

We observe that the cut-offs $\chi,\zeta$ have no influence on the restriction of $T\alpha$ or $\bar T \alpha$ to $B$. Therefore $T$ and $\bar T$ coincide as bounded operators $L^1(B')\cap d_c^{-1}(L^1(B'))\to L^q(B)$.

Thus we are left with the proof of a) and b). To this end, we must deal separately with the case when
 degree of $\omega$ equals $n+1$.
}

With our previous notations, if the degree of $\phi$ is different from $n$ (i.e. if the degree of $\omega$
 is different from $n+1$), then $[d_c,\zeta]$ is a linear operator of order $0$ with coefficients compactly supported in $B_0$. Therefore
 $$
 \omega   = \zeta d_c\phi +   [d_c,\zeta]\phi                   \in L^1(\he n).
 $$
  Thus, we can apply Lemma \ref{by parts} to $\psi:=\zeta\phi$ and we conclude that 
$$
\omega\in L_0^1(\he n)\cap \mathrm{ker}(d).
$$

{ 
Therefore, by Theorem \ref{core},
$K_0\omega\in L^q(\he n)$, where $q=Q/(Q-1)$. Let us prove that the same assertion holds for
$K_{0,R}\omega$ and hence $\phi\in L^q(B')$. In fact,
\begin{equation*}\begin{split}
K_{0,R}\omega(x) &=  K_0\omega(x) + (\psi_R - 1)k_0 \ast\omega(x)
\\& =
K_0\omega(x) + \int_{\he n}(\psi_R - 1)k_0(y^{-1}x)  \omega(y)\, dy.
\end{split}\end{equation*}

Notice now that $(\psi_R - 1)k_0$ is a smooth function and that $y^{-1}x$ lies in a compact set
when $x\in B'$ (since $\omega$ is compactly supported). Thus 
$$
\| (\psi_R - 1)k_0 \ast\omega\|_{L^q(B')}\le C\,\| (\psi_R - 1)k_0 \|_{L^{q}(B')}\|\omega\|_{L^{1}(B')}\leq C'\,\|\omega\|_{L^{1}(B')}.
$$
}

Suppose now the degree of $\phi$ equals $n$.

Then, by Lemma \ref{leibniz} $ [d_c,\zeta]$, is the sum of
a linear operator  $P_0(W^2\zeta)$ of order $0$ with coefficients
 compactly
 supported in $B_0$, and of a linear operator $P_1(W\zeta)$  of order $1$ with coefficients
 compactly
 supported in $B_0$. As above,
  $$
 \omega  = \zeta d_c\phi +   [d_c,\zeta]\phi   =    
  \zeta d_c\phi +    P_0(W^2\zeta)\phi + P_1(W\zeta)   \phi.   
 $$
 Again
 $$
   \zeta d_c\phi +    P_0(W^2\zeta)\phi   \in L^1(\he n).
$$
{ 
Notice now that, if $\phi$ has degree $n$, then 
$$
P_1(W\zeta)   \phi = P_1(W\zeta)  K_{0,R}(\chi\alpha).
$$
But $\alpha$
has degree $n+1$, so that, by Lemma \ref{L1 boundedness of Kd new}, it is associated with a kernel of type 2.
Thus, by Lemma \ref{truncation} and Corollary \ref{marc alternative coroll}, and keeping in mind that
the coefficients of $P_1$ are compactly supported in $B_0$ 

$$
P_1(W\zeta)   \phi \in L^1(\he n),
$$
so that again
 $$
 \omega  \in L^1(\he n)
 $$
 and
\begin{equation}\label{aug 6}
\| \omega\|_{L^1(\he n)} \le C\|\alpha\|_{L^1(B')}.
\end{equation}

Again, we can apply Lemma \ref{by parts} to $\psi:=\zeta\phi$ and we conclude that 
$$
\omega\in L_0^1(\he n)\cap \mathrm{ker}(d).
$$
This proves a). On the other hand,
 by Theorem \ref{core},
$K_0\omega\in L^q(\he n)$, where $q=Q/(Q-2)$. Arguing as above, the same assertion holds for
$K_{0,R}\omega$ and b) follows keeping in mind \eqref{aug 6}.}

{  We observe that $T\omega=K_{0,R}(\chi\omega)$, where $K$ has compactly supported kernel $\psi_R k_0$, with $k_0$ of type $1$ (resp. type $2$ if $h=n+1$). If $h\not=n+1$, Lemma \ref{berthollet 1} and Theorem \ref{berthollet 3} apply. If $h=n+1$, 
\begin{align}\label{berthollet 4}
WT\omega
&= W(\chi\omega\ast\psi_R k_0)=\chi\omega\ast W(\psi_R k_0)\\
&=\chi\omega\ast(W\psi_R)k_0+\chi\omega\ast\psi_R(Wk_0).
\end{align}
Lemma \ref{berthollet 1} and Theorem \ref{berthollet 3} apply to both terms. They provide an $r>0$ and a $p>1$ such that $T$ (resp. $WT$) maps $L^1(\he{n})$ to $W^{r,p}_{\mathrm{loc}}(\he{n})$ and $W^{s,p}(\he{n})$ to $W^{s+1,p}_{\mathrm{loc}}(\he{n})$.
}

\end{proof}

\subsection{Composition of homotopies}

This is the final step which provides an inverse to $d_c$ on $d_c$-closed $L^1$ forms defined on a ball, as stated in Theorem \ref{poincareintro}.

\begin{corollary}[Interior Poincar\'e and Sobolev inequalities]\label{poincare}
Let $B\Subset B'$ be concentric balls in $\he n$. For $h=1,\ldots,2n$, let $q=Q/(Q-1)$ if $h\not=n+1$ and $q=Q/(Q-2)$ if $h=n+1$. For every $d_c$-closed $h$-form $\alpha\in L^1(B',E_0^h)$, there exists an $h-1$-form $\phi\in L^q(B,E_0^{h-1})$, such that 
$$
d_c\phi=\alpha_{|B}\qquad\mbox{and}\qquad \|\phi\|_{L^q(B,E_0^{h-1})}\leq C\,\|\alpha\|_{L^1(B',E_0^h)}.
$$
Furthermore, if $\alpha$ is compactly supported, so is $\phi$.
\end{corollary}

\begin{proof}
Proposition \ref{smoothing} allows to replace $\alpha$ with $S\alpha$ whose first 3 derivatives, in $L^2$ norm, are controlled by $\|\alpha\|_1$. Then $\beta=\Pi_E(S\alpha)$ and its 2 first derivatives are controlled by $\|\alpha\|_1$, and $d\beta=0$. Apply Iwaniec-Lutoborski's  homotopy \cite{IL} to get a differential $(h-1)$-form $\gamma$ on $B$ such that $d\gamma=\beta$ and with 2 first derivatives controlled by $\|\alpha\|_1$ in $L^{2}$ (IL's homotopy is an operator of type $1$). The Euclidean Sobolev inequality implies that $\|\gamma\|_q$ is controlled by $\|\alpha\|_1$, for $q=Q/(Q-2)$. \emph{A fortiori}, for $q=Q/(Q-1)$. So is $\|\phi\|_q$, where $\phi=\Pi_{{E}_0}\gamma$ satisfies
\begin{align*}
d_c\phi&=\Pi_{{E}_0}d\Pi_E \Pi_{{E}_0}\gamma\\
&=\Pi_{{E}_0}d\gamma\\
&=\Pi_{{E}_0}\Pi_E S\alpha\\
&=S\alpha.
\end{align*}
This proves the interior Poincar\'e inequality. 

Allowing $S$ to win $4$ derivatives instead of $3$ provides a control on the $W^{2,p}$ norm of $\beta$ for some $p>1$. This allows to replace Iwaniec-Lutoborski's homotopy \cite{IL} with Mircea-Mircea-Monniaux' homotopy \cite{mitrea_mitrea_monniaux} which preserves compactly supported forms. When $\alpha$ is compactly supported, so are $S\alpha$, $\beta$, the primitive $\gamma$ provided by Mircea-Mircea-Monniaux, and $\phi$. This proves a Sobolev inequality.
\end{proof}

{  
\begin{remark} Without loss of generality, in Corollary \ref{poincare} we can assume that $d^*_c\phi =0$,
provided we replace $B$ by a smaller ball $\tilde B\Subset B$.
Indeed, let $\psi$ be a cut-off function, $\psi\equiv 1$ on $\tilde B$ and $\supp \psi\subset B$. Set
$$
\tilde\phi := d^*_c d_c \Delta_{\mathbb H}^{-1} (\psi \phi).
$$
Obviously, $d^*_c \tilde\phi =0$.
Since $d^*_cd_c \Delta_{\mathbb H}^{-1}$ is associated with a kernel of type 0 and $q>1$, we have
\begin{equation*}\begin{split}
\| \tilde\phi  &\|_{L^q(\tilde B,E_0^{h-1})} \le \| \tilde\phi\|_{L^q(\he n,E_0^{h-1})}\
\\&
\le C  \|\psi\phi\|_{L^q(\he n,E_0^{h-1})} \le C  \|\phi\|_{L^q(B,E_0^{h-1})}
\le C  \|\alpha\|_{L^q(B',E_0^{h})}.
\end{split}\end{equation*}

Notice now that
$$
\psi \phi = d_cd^*_c \Delta_{\mathbb H}^{-1} (\psi \phi) + d^*_cd_c \Delta_{\mathbb H}^{-1} (\psi \phi)
= d_cd^*_c \Delta_{\mathbb H}^{-1} (\psi \phi) + \tilde\phi.
$$
Thus in $\tilde B$
$$
d_c \tilde\phi  = d_c (\psi\phi) =d_c \phi = \alpha.
$$
\end{remark}
}

The following globalization procedure, established for spaces of $L^p$ differential forms, $p>1$, in \cite{BFP2}, extends to $L^1$.

\section{Bounded geometry Riemannian and contact manifolds}\label{bounded geometry}

A contact structure on an  odd-dimensional manifold $M$ is a smooth distribution of hyperplanes $H$ which is maximally 
nonintegrable in the following sense: if $\theta$ is a locally defined smooth 1-form such that $H=\mathrm{ker}(\theta)$, then $d\theta$ restricts to a non-degenerate 2-form on $H$,
{ i.e. if $2n+1$ is the dimension of $M$, then $\theta\wedge(d\theta)^n\neq 0 $ on $M$ (see \cite{mcduff_salamon}, Proposition 3.41).}  A contact manifold $(M,H)$ is the data of a smooth manifold $M$ and a contact structure $H$ on $M$. 

Contact diffeomorphisms  are contact structure preserving diffeomorphisms between contact manifolds. 

We recall that, by a classical {  theorem of} Darboux, any contact manifold $(M,H)$ is locally 
contact diffeomorphic to the Heisenberg group $\he n$ (see \cite{mcduff_salamon}, p. 112).

We recall that the construction of Rumin's complex can be carried out for general contact manifolds (see, e.g. \cite{rumin_jdg},
\cite{rumin_palermo}) yielding a complex of differential forms - still denoted by $(E_0^\bullet, d_c)$ - such that
\begin{itemize}
\item[i)] $d_c^2=0$;
\item[ii)] the complex $(E_0^\bullet,d_c)$ is homotopically equivalent to the de Rham complex 
$\Omega:= (\Omega^\bullet,d)$. Thus,  if $D\subset \he n$ is an open set, unambiguously we write $H^h(D)$
for the $h$-th cohomology group;
\item[iii)] $d_c: E_0^h\to E_0^{h+1}$ is a homogeneous differential operator in the 
horizontal derivatives (i.e. derivatives along $H$)
of order 1 if $h\neq n$, whereas $d_c: E_0^n\to E_0^{n+1}$ is an homogeneous differential operator 
of order 2 in the 
horizontal derivatives.
\end{itemize}

Moreover, if $\phi $ is a contactomorphism  from an open set $\mc U\subset \he{n}$ to $M$, and we denote by $\mc V$
the open set  $\mc V:= \phi(\mc U)$, we have
\begin{itemize}
\item[i)] $\phi^\# E_0^\bullet(\mc V) = E_0^\bullet(\mc U) $;
\item[ii)] $d_c\phi^\# = \phi^\# d_c$.
\item[iii)] if $\zeta$ is a smooth function in $M$, then the differential operator in $\mc U\subset \he n$ defined by $v \to \phi^\#[d_c,\zeta] (\phi^{-1})^\#v$ is a differential operator of order
zero if $v\in E_0^h(\mc U)$, $h\ne n$ and a differential operator of order
1 if $v\in E_0^n(\mc U)$
\end{itemize}
(see \cite{BFP2}, Proposition 3.11).
\medskip

If a Riemanniam metric $g$ is defined on $H$, we refer to the $(M,H,g)$ as to a sub-Riemannian contact manifold.

In turn, in any sub-Riemannian contact manifold $(M,H,g)$ we can define a sub-Rieman\-nian distance $d_M$ (see e.g. \cite{montgomery}) inducing on $M$ the same topology of $M$ as a manifold. In particular, Heisenberg groups can be viewed as sub-Riemannian contact manifolds. If we choose on the contact sub-bundle of $\he n$ a left-invariant metric, it turns out that the associated sub-Riemanian metric is left-invariant, too.

\subsection{Bounded geometry and controlled coverings}

We give now the definition of Riemannian manifold of bounded geometry as well as the definition of
contact manifold of bounded geometry.

\begin{definition}\label{bded geometry bis} Let {  $k$} be a positive integer and
let $B(0,1)$ denote the unit  ball in $\mathbb R^n$.  We say that a  Riemannian manifold $(M,g)$ has \emph{bounded $C^k$-geometry} is there exist constants $r, C>0$ such that, for every $x\in M$, 
there exists a  diffeomorphism preserving  $\phi_x : B(0,1)\to M$ {  that satisfies}
\begin{enumerate}
  \item $B(x,r)\subset\phi_x(B(0,1))$;
  \item $\phi_x$ is $C$-bi-Lipschitz, i.e.
  \begin{equation}\label{bilipR}
\frac1C|p-q| \le  d_M(\phi_x (p), \phi_x(q)) \le C|p-q| \qquad \mbox{for all $p,q\in B(0,1)$};
\end{equation}
  \item coordinate changes $\phi_x\circ\phi_y^{-1}$ and their first $k$ derivatives  are bounded by $C$.
\end{enumerate}
\end{definition}

The counterpart of the above definition for subRiemannian contact manifolds reads as follows:

\begin{definition}\label{contact bis} Let {  $k$} be a positive integer and
let $B(e,1)$ denote the unit subRiemannian ball in $\he n$.  We say that a subRiemannian contact manifold $(M,H,g)$ has \emph{bounded $C^k$-geometry} is there exist constants $r, C>0$ such that, for every $x\in M$, 
there exists a contactomorphism (i.e. a diffeomorphism preserving the contact structure) $\phi_x : B(e,1)\to M$ that satisfies
\begin{enumerate}
  \item $B(x,r)\subset\phi_x(B(e,1))$;
  \item $\phi_x$ is $C$-bi-Lipschitz, i.e.
  \begin{equation}\label{bilip}
\frac1Cd(p,q)\le  d_M(\phi_x (p), \phi_x(q)) \le Cd(p,q)\qquad \mbox{for all $p,q\in B(e,1)$};
\end{equation}
  \item coordinate changes $\phi_x\circ\phi_y^{-1}$ and their first $k$ derivatives with respect to unit left-invariant horizontal vector fields are bounded by $C$.
\end{enumerate}
\end{definition}

In \cite{BFP2}, Lemma 5.10, we proved the following covering lemma (that is basically  \cite{mattila}, Theorem 1,2). We state
it for subRiemannian contact manifolds, but it still holds in the Riemannian setting.

\begin{lemma}\label{covering} Let $(M,H,g)$ be a bounded $C^k$-geometry subRiemannian contact manifold,
where $k$ is a positive integer. Then there
exists $\rho>0$ (depending only on the radius $r$ of Definition \ref{contact bis}) and an at most countable covering $\{ B(x_j, \rho)\}$ of $M$ such that
\begin{itemize}
\item[i)] each ball $B(x_j, \rho)$ is contained in the image of one of the contact charts of Definition \ref{contact bis};
\item[ii)] $B(x_j, \frac15\rho)\cap B(x_i, \frac15\rho)=\emptyset$ if $i\neq j$;
\item[iii)] the covering is uniformly locally finite. Even more,  there exists a  $N=N(M)\in \mathbb N$ such that 
for each ball $B(x,\rho)$ 
$$
\#\{k\in \mathbb N \mbox{ such that } B(x_k,\rho)\cap B(x,\rho)\neq\emptyset\} \le N.
$$
In addition, if $B(x_k,\rho)\cap B(x,\rho)\neq\emptyset$, then $B(x_k,\rho)\subset B(x,r)$, where $B(x,r)$ has
been defined in Definition \ref{contact bis}-(2));
\item[iv)] all balls $B(x_k,\rho)$ have comparable measures.
\end{itemize}

\end{lemma}

\subsection{Sobolev spaces of Rumin forms on contact manifolds}

A key feature of Rumin's complex for Heisenberg groups is its invariance under smooth 
contactomorphisms: if $U$ and $V$ are open subsets of $\he{n}$ and $\phi:U\to V$ is a contact 
structure preserving diffeomorphism, then $\phi$ pulls back Rumin forms. We use the same 
notation $\phi^{\#}$ as for the pull-back of usual differential forms. We use this to define 
Sobolev spaces on bounded geometry contact subRiemannian manifolds. They will be needed in the construction of global smoothing homotopies, Proposition \ref{bddgeometry}.

In the Riemannian setting, Sobolev spaces of differential forms are invariant with respect to the 
pull-back operator associated with sufficiently smooth diffeomorphisms (see, e.g. \cite{schwarz}, Lemma 1.3.9).  
An analogous statement holds for Folland-Stein Sobolev spaces in Heisenberg groups,
provided we restrict ourselves to contact diffeomorphisms. Indeed we have:

\begin{lemma}\label{pullback}
If $k$ is a positive integer, let $U, V\subset \he n$ be connected open extension subsets of $\he n$ (see Definition \ref{extension domain}). 
 Let $U_0,V_0$ be open neighborhoods of
$U$ and $V$, respectively, and let $\phi:U_0\to V_0$ be a $C^k$-bounded contact diffeomorphism such that
$\phi (U) \subset V$. If $p>1$ and $s$ is a real number, $0\leq s\leq k-1$
then the pull-back operator $\phi^{\#}$ from $W^{s,p}(V,E_0^\bullet) $  to $W^{s,p}(U,E_0^\bullet)$  is bounded, 
and its norm depends only on the $C^k$ norms of $\phi$ and $\phi^{-1}$. This extends to $p=1$ if $s$ is an integer.
\end{lemma}

\begin{proof}
Consider the case $p>1$. The proof for the case $p=1$ is analogous but shorter, since
we do not need interpolation arguments. 
 Let $\psi\in\mc D(\he{n})$ be a cut-off function supported in $V_0$, $\psi\equiv 1$ on $V$. If $u\in  \mc D(\he{n}, E_0^\bullet)$, then
$\phi^{\#} (r_{V_0}(\psi u))$ is well defined and supported in $U_0$, so that can be continued by zero outside $U_0$.
Denote by $\big(\phi^{\#} (r_{V_0}(\psi u))\big)_0$ this extension.
Suppose now $u\in \mc D(\he{n}, E_0^\bullet)$, and consider the map
$$
u\to L(u):= \big(\phi^{\#} (r_{V_0}(\psi u))\big)_0.
$$
If $s$ is an integer, by the chain rule and our assumptions on $\phi$
$$
\| \big(\phi^{\#} (r_{V_0}(\psi u))\big)_0\|_{W^{s,p}(\he n, E_0^\bullet)} \le C \|u\|_{W^{s,p}(\he{n}, E_0^\bullet)}.
$$
Thus, by density and interpolation,
$L$ is a bounded linear operator from $W^{s,p}(\he{n},E_0^\bullet)$ to $W^{s,p}(\he{n},E_0^\bullet)$ for $s\ge 0$.

Take now $\alpha\in W^{s,p}(V,E_0^\bullet) $, an let $\tilde\alpha\in W^{s,p}(\he{n},E_0^\bullet)$ an arbitrary extension of $\alpha$
outside $V$.  We notice  that $\big(\phi^{\#} (r_{V_0}(\psi \tilde\alpha)\big)_0$ is an
extension of $\phi^{\#} (\alpha)$ outside $U$. Indeed,  if $x\in U$ (and therefore $\phi(x)\in V$)
and $v_1,\dots,v_\bullet$ are tangent vectors at $x$, we have 
\begin{equation*}\begin{split}
\big(\phi^{\#} &  (r_{V_0}(\psi \tilde\alpha)\big)_0(x) (v_1,\dots,v_\bullet) 
\\&=
\phi^{\#} (r_{V_0}(\psi \tilde\alpha)(x)(v_1,\dots,v_\bullet)
\\&=
r_{V_0}(\psi \tilde\alpha)(\phi(x))(d\phi(x)v_1,\dots,d\phi(x)v_\bullet)
\\&=
\psi \tilde\alpha (\phi(x)) (d\phi(x)v_1,\dots,d\phi(x)v_\bullet)
\\&=
\alpha (\phi(x)) (d\phi(x)v_1,\dots,d\phi(x)v_\bullet)
\\&=
\phi^{\#} (\alpha)(x) (v_1,\dots,v_\bullet).
\end{split}\end{equation*}
Then
\begin{equation*}\begin{split}
\| \phi^{\#} &  (\alpha) \|_{W^{s,p}( UE_0^\bullet)}
\le
\| \big(\phi^{\#} (r_{V_0}(\psi \tilde\alpha)\big)_0 \|_{W^{s,p}( \he{n}, E_0^\bullet)}.
\le C \| \tilde\alpha \|_{W^{s,p}( \he{n}, E_0^\bullet)}.
\end{split}\end{equation*}
Taking the infimum of the right-hand side of this inequality for all extensions $\tilde\alpha\in W^{s,p}( \he{n}, E_0^\bullet)$
of $\alpha$, the assertion follows.

\end{proof}

\begin{definition}\label{carte} 
Let $k$ be a positive integer, and let $(M,H,g)$ be a bounded $C^k$-geometry subRiemannian contact manifold,  and 
let $\{\chi_j\}$ be a partition of unity subordinate to the atlas $\mc U:=\{ B(x_j, \rho), \phi_{x_j}\}$ of Lemma \ref{covering}. 
From now on, for the sake of
simplicity, we shall write $\phi_j:=\phi_{x_j}$. We stress that $\phi_j^{-1}(\mathrm{supp}\;\chi_j) \subset B(e,1)$.
Fix $a\geq 1$, $p\geq 1$ and $s\in \mathbb R$, $0\leq s\leq k-1$. If $\alpha$ is a Rumin differential form on $M$, we say that $\alpha\in \ell^a(W^{s,p})_{\mc U}(M, E_0^\bullet)$ if
$$
\phi_j^\# (\chi_j\alpha) \in W^{s,p}(\he n, E_0^\bullet)\quad\mbox{for $j\in \mathbb N$}
$$
(notice that 
$\phi_j^\# (\chi_j\alpha)$ 
is compactly supported in $B(e,1)$ and therefore can be continued by zero on
all of $\he n$) and the sequence $\|\phi_j^\# (\chi_j\alpha)\|_{W^{s,p}(\he n, E_0^\bullet)}^a$ is summable. Then we set
$$
\|\alpha\|_{\ell^a(W^{s,p})_{\mc U}(M, E_0^\bullet)}:= \left(\sum_j \|\phi_j^\# (\chi_j\alpha)\|_{W^{s,p}(\he n, E_0^\bullet)}â\right)^{1/a}.
$$
\end{definition}

Obviously, the same definition can be formulated for bounded $C^k$-geometry Riemannian manifolds. One recovers global $W^{s,p}$ spaces of $\R^n$ and $\he{n}$ by taking $a=p$.

The following result shows that the definition of the Sobolev spaces $\ell^a(W^{s,p})_{\mc U}(M, E_0^\bullet)$ do not
depend on the atlas $\mc U$.  An analogous statement holds in the Riemannian setting.
Therefore, once the proposition is proved, we drop the index $\mc U$ from the notation for Sobolev norms.

\begin{proposition} Let $k$, $a$, $p$ and $s$ be as above, and let $(M,H,g)$ be a bounded $C^k$-geometry subRiemannian contact manifold. If $\mc U':=\{ B(y_j, \rho'), \phi_{y_j}'\}$ is another
atlas of $M$ satisying Definition \ref{contact bis} and Lemma \ref{covering} with the same choice of $\rho$, and $\{\chi_j'\}$ is an associated partition of
unity, then
$$
\ell^a(W^{s,p})_{\mc U}(M, E_0^\bullet)= \ell^a(W^{s,p})_{\mc U'}(M, E_0^\bullet),
$$
with equivalent norms. 
\end{proposition}

\begin{proof} 
Let $j\in \mathbb N$ be fixed, and let $(B(x_j,\rho),\phi_j)$ be  a chart of $\mc U$. We can write
$$
\chi_j = \sum_{k\in I_j} \chi_k' \chi_j,
$$
where $\# I_j \le N$, since, by Lemma \ref{covering} iii), $B(x_j,\rho)$ is covered by at most
$N$ balls of the covering associated with $\mc U'$. Thus,
by Definition \ref{contact bis}-(3)
and keeping in mind that $\mathrm{supp}\, \chi_k' \subset B(x_j,r)$ (since $3\rho<r$), we have
\begin{equation*}\begin{split}
\|\phi_j^\# & (\chi_j\alpha)\|_{W^{s,p}(\he n, E_0^\bullet)} \le 
\sum_{k\in I_j} \|\phi_j^\# (\chi_k' \chi_j \alpha)\|_{W^{s,p}(\he n, E_0^\bullet)}
\\&
\le c\sum_{k\in I_j} \|\phi_j^\# (\chi_k' \alpha)\|_{W^{s,p}(\he n, E_0^\bullet)}
\\&
= c\sum_{k\in I_j} \|(\phi_j\phi_k'^{-1})^\# \phi_k'^\# (\chi_k'  \alpha)\|_{W^{s,p}(\he n, E_0^\bullet)}
\\&
\le 
c \sum_{k\in I_j} \| \phi_k'^\# (\chi_k' \alpha)\|_{W^{s,p}(\he n, E_0^\bullet)}
\\&
\le cN \|\alpha\|_{W^{s,p}_{\mc U'}(M, E_0^\bullet)}.
\end{split}\end{equation*}
A similar inequality holds for $a$-th powers, since the number of terms in the sum is bounded.
\end{proof}

\section{Smoothing homotopies on bounded geometry (contact) manifolds}\label{SH}

\subsection{Proof of Theorem \ref{bddgeometryintro}}

Here, we piece together local smoothing homotopies using contact charts and a partition of unity. The formula for the global smoothing operator $S$ mixes local smoothing operators $S$ and homotopies $T$, therefore the gain in differentiability is less than $1$. It needs be measured in terms of fractional Sobolev spaces. Iterating the initial operator allows to gain arbitrarily large numbers of derivatives.

\begin{proposition}[Global smoothing homotopies]\label{bddgeometry}

Let $k\ge 3$ be an integer index, and let $M$ be a subRiemannian contact manifold of dimension $2n+1$ 
and bounded $C^k$-geometry. For $h=1,\ldots,2n$, let $q=Q/(Q-1)$ if 
$h\not=n+1$ and $q=Q/(Q-2)$ if $h=n+1$. Let $1\leq q'\leq q$. There exist an operator 
$T_M$ on $h$-forms on $M$ which is bounded from 
$L^{1}(M,E_0^\bullet)\cap d^{-1}L^1(M,E_0^\bullet)$ to $L^{q'}(M,E_0^\bullet)$ and an operator $S_M$ which is 
bounded from 
$L^{1}(M,E_0^\bullet)\cap d^{-1}L^1(M,E_0^\bullet)$ to $W^{k-1,q'}(M,E_0^\bullet)$ such that $1=S_M+d_c T_M+T_M d_c$. 
\end{proposition}

\begin{proof}
The global operators $S_M$ and $T_M$ are obtained in two steps. First, 
one transports by charts $\phi_j$ the local operators $S$ and $T$ 
constructed on Heisenberg balls in Proposition \ref{smoothing} and one 
pieces them together using a controlled partition of unity $\{\chi_j\}$. Note that the following formulae differ from those of
\cite{BFP2}, section 7.
$$ {\bf T}u:= \sum_j  \chi_j  \big((\phi_j^{-1})^\# 
( T ( \phi_j^\# (u_{|10B_j}))_{{B'}_\mathbb H})_{{B}_\mathbb H}\big)_{|B_j},
$$
\begin{align*}\label{S}
{\bf{S} }u:&= \sum_j  \chi_j  \big((\phi_j^{-1})^\# 
( S ( \phi_j^\# (u_{|10B_j}))_{{B'}_\mathbb H})_{{B}_\mathbb H}\big)_{|B_j} \\&-
\sum_j [\chi_j ,d_c] \big((\phi_j^{-1})^\# 
( T ( \phi_j^\# (u_{|10B_j}))_{{B'}_\mathbb H})_{{B}_\mathbb H}\big)_{|B_j}.\end{align*}
In these formulae, $u$ is a Rumin form defined globally on $M$. The chart $\phi_j$ is defined on the larger Heisenberg ball $B'$, it maps it into $10B_j$. The image of the smaller Heisenberg ball $B'$ contains $B_j$. Therefore $T$ can be applied to the pulled-back form $\phi_j^\# (u)$ and the form $T\phi_j^\# (u)$, which depends only on the restriction of $u$ to $10B_j$, is defined on all of $B$. Its push-forward to $M$ is defined on $B_j$. The product of this form with $\chi_j$ has compact support in $B_j$. Therefore the sum is locally finite (only boundedly many terms do not vanish at a given point). In the sequel, the notation will be abbreviated as
\begin{equation}\label{T}
 {\bf T}u:= \sum_j  \chi_j  (\phi_j^{-1})^\# 
 T  \phi_j^\# (u)
\end{equation}
and 
\begin{equation}\label{S}
{\bf{S} }u:= \sum_j  \chi_j  (\phi_j^{-1})^\# S\phi_j^\# (u) -
\sum_j [\chi_j ,d_c]  (\phi_j^{-1})^\# 
  T \phi_j^\#(u).
\end{equation}
Second, one iterates $\bf{S}$, i.e. one sets $S_M=\bf{S}^{\ell}$ for $\ell$ large enough.

Given a function space $F$ of forms on the unit Heisenberg ball, let us denote by $\ell^a(F)$ the space of differential forms $\omega$ on $M$ such that the sequence $\|\phi_j^{\#}\omega_{|B_j}\|_F$ belongs to $\ell^a$. 

Since the covering has bounded multiplicity, 
$$
\ell^1(L^1(M,E_0^\bullet))= L^1(M,E_0^\bullet)
$$
and 
$$
\ell^1(L^1(M,E_0^\bullet)\cap d_c^{-1}(L^1(M,E_0^\bullet)))=L^1(M,E_0^\bullet)\cap d_c^{-1}(L^1(M,E_0^\bullet)).
$$
 Indeed, let us prove (for instance) the first equality.
If $N$ is an upper bound for the multiplicity of the covering $\{10B_i\}$, for every form $u$, 
$$
\|u\|_{\ell^1(L^1(M,E_0^\bullet))}=\sum_j \|u_{|10B_j}\|_{L^1(10B_j, E_0^\bullet)}\leq N\|u\|_{L^1(M,E_0^\bullet)}.
$$

Let us show that ${\bf S}$ and ${\bf T}$ win a bit of differentiability:
\begin{itemize}
  \item ${\bf S}$ and ${\bf T}:\ell^1(L^1(M,E_0^\bullet)\cap d_c^{-1}(L^1(M,E_0^\bullet)))\to \ell^1(W^{r,p}(M,E_0^\bullet))$
   are bounded for some $r>0$ and some $p>1$;
  \item for all $1\leq s\leq k-1$, ${\bf S}$ and ${\bf T}:\ell^1(W^{s-1,p}(M,E_0^\bullet))\to \ell^1(W^{s,p}(M,E_0^\bullet))$ are bounded;
  \item for all $0\leq s\leq k-1$, ${\bf T}d_c$ and $d_c{\bf T}:\ell^1(W^{s,p}(M,E_0^\bullet))\to \ell^1(W^{s,p}(M,E_0^\bullet))$ are bounded.
\end{itemize}

First, let us understand local continuity properties. In the expressions for ${\bf S}$, ${\bf T}$, $d_c{\bf T}$ and ${\bf T}d_c$, we find the following types of terms:
\begin{align}
\chi_j  (\phi_j^{-1})^\# S\phi_j^\# &=  (\phi_j^{-1})^\#(\chi S)\phi_j^\#,\\
[\chi_j ,d_c]  (\phi_j^{-1})^\#  T \phi_j^\# &= (\phi_j^{-1})^\#([\chi,d_c] T)\phi_j^\#,\\
\chi_j  (\phi_j^{-1})^\#  T  \phi_j^\# &= (\phi_j^{-1})^\# (\chi T)  \phi_j^\# ,\\
\chi_j  (\phi_j^{-1})^\#  T  \phi_j^\# d_c &= (\phi_j^{-1})^\# (\chi Td_c)  \phi_j^\# ,\\
\chi_j d_c (\phi_j^{-1})^\#  T  \phi_j^\# &= (\phi_j^{-1})^\# (\chi d_c T)  \phi_j^\# ,
\end{align}
where $\chi=\chi_j\circ\phi_j$. From Theorem \ref{berthollet 2}, we know that multiplication by a function $\chi\in\mathcal{D}$ is a bounded operator on all Sobolev spaces $W^{s,p}$, with norm depending on the size of horizontal derivatives of $\chi$ only. Since functions $\chi_j\circ\phi_j$ have uniformly horizontal bounded derivatives, we can ignore them in the sequel.

Proposition \ref{smoothing} takes care of terms of the form $S$, $T$, $Td_c$ and $d_c T$. Only $[\chi,d_c]T$ remains. If $h\neq n+1$, then the commutator has order zero
and $[\chi,d_c]T$ can be written as a linear combination of components of $T$ multiplied by smooth compactly
supported functions. If $h=n+1$, then the commutator has order $1$ and $[\chi,d_c]T$ can be written as a linear combination
of horizontal derivatives composed with components of $T$, multiplied by smooth compactly
supported functions. Keeping in mind Theorem \ref{berthollet 2}, we can apply Proposition \ref{smoothing} in both cases, and conclude that all types of terms correspond to operators on the Heisenberg ball which are bounded as required. 

By construction, since the covering has bounded multiplicity and derivatives of cut-offs and charts are controlled uniformly, summing up each type of term gives bounded operators from $\ell^1(L^1(M,E_0^\bullet)\cap d_c^{-1}(L^1)(M,E_0^\bullet))$ to $\ell^1(L^q(M,E_0^\bullet))$ or to $\ell^1(W^{r,p}(M,E_0^\bullet))$ for some $r>0$ and $p>1$, and $\ell^1(W^{s-1,p}(M,E_0^\bullet))\to \ell^1(W^{s,p}(M,E_0^\bullet))$ or $\ell^1(W^{s,p}(M,E_0^\bullet))\to \ell^1(W^{s,p}(M,E_0^\bullet))$, as announced.

By construction, ${\bf S}+d_c {\bf T}+{\bf T}d_c=1$, hence $d_c{\bf S}={\bf S}d_c$.

When iterating, we write ${\bf S}^\ell=1-d_cT_\ell-T_\ell d_c$. The recursion formula is $T_{\ell+1}=T_\ell +\mathbf{T}-d_c T_\ell \mathbf{T} -T_\ell d_c \mathbf{T}$. 

Let us show by induction on $\ell$ that
\begin{itemize}
  \item $T_\ell$ maps $\ell^1(L^1(M,E_0^\bullet)\cap\, d_c^{-1}(L^1(M,E_0^\bullet)))$ to $\ell^1(L^q(M,E_0^\bullet))$ and to 
	
	\noindent$\ell^1(W^{r,p}(M,E_0^\bullet))$ for some $r>0$ and $p>1$.
  \item $d_c T_\ell$ and $T_\ell d_c$ are bounded on $\ell^1(W^{s,p}(M,E_0^\bullet))$ for all $s\leq k-1$ and $p>1$.
\end{itemize}
Note that $T_1={\bf T}$. We have just shown that $d_c T_1$ and $T_1 d_c$ are bounded on $\ell^1(W^{s,p})$ and $T_1$ maps $\ell^1(L^1(M,E_0^\bullet)\cap\, d_c^{-1}(L^1(M,E_0^\bullet)))$ to $\ell^1(L^q(M,E_0^\bullet))$. Assume that $T_\ell$ does as well. The induction formula
$$
d_cT_{\ell+1}=d_c T_\ell +d_c T_1 -d_cT_\ell d_c T_1,\quad
T_{\ell+1}d_c=T_\ell d_c +T_1 d_c-d_c T_\ell T_1 d_c -T_\ell d_c T_1 d_c.
$$
shows that $d_c T_{\ell+1}$ and $T_{\ell+1} d_c$ are bounded on $\ell^1(W^{s,p}(M,E_0^\bullet))$. This implies that $T_{\ell+1}$ maps $\ell^1(L^1(M,E_0^\bullet)\cap\, d_c^{-1}(L^1(M,E_0^\bullet)))$ to $\ell^1(L^q(M,E_0^\bullet))$ and to $\ell^1(W^{r,p}(M,E_0^\bullet))$ for some $r>0$ and $p>1$, and completes the induction proof. For $\ell$ larger enough, $S_M:={\bf S}^\ell$ maps $\ell^1(L^1(M,E_0^\bullet)\cap d_c^{-1}(L^1(M,E_0^\bullet)))$ to $\ell^1(W^{k-1,q}(M,E_0^\bullet))$.

Finally, if $1\leq q'\leq q$, $\ell^{1}\subset \ell^{q'}$ and $L^{q}_{\mathrm{loc}}\subset L^{q'}_{\mathrm{loc}}$, hence 
$$
\ell^1(L^{q}(M,E_0^\bullet))\subset \ell^{q'}(L^{q'}(M,E_0^\bullet))=L^{q'}(M,E_0^\bullet).
$$
This completes the proof that $T_M:=T_\ell$, $\ell$ large enough, maps 
$L^1(M,E_0^\bullet)\cap\, d_c^{-1}(L^1(M,E_0^\bullet))$ to $L^{q'}(M,E_0^\bullet)$ and $S_M$ maps $L^1(M,E_0^\bullet)\cap d_c^{-1}(L^1(M,E_0^\bullet))$ to $W^{k-1,q'}(M,E_0^\bullet)$.

\end{proof}

\subsection{Application to geometric group theory}

According to \cite{PPcup}, such smoothing homotopies are the necessary ingredient in order to prove that Rumin's complex can be used to compute the $\ell^{q,1}$-cohomology of a subRiemannian contact manifold. We shall not define this quasiisometry invariant of groups here, but merely state a consequence of Theorems \ref{poincareglobal}, \ref{poincareintro} and \ref{bddgeometryintro} for geometric group theory.

\begin{corollary}[$\ell^{q,1}$-cohomology of Heisenberg groups]\label{GGT}
For $h=0,\ldots,2n$, let $q=Q/(Q-1)$ if $h\not=n+1$ and $q=Q/(Q-2)$ if $h=n+1$. 
Then $\ell^{q,1}H^h(\he n)$ is finite dimensional.
\end{corollary}

\begin{proof}
\cite{PPcup} asserts that for all subRiemannian contact manifolds $M$ of $C^3$-bounded geometry, and all $q\geq 1$, $\ell^{q,1}H^h(M)$ is isomorphic to the quotient of the space of $d_c$-closed $h$-forms by the image of $d_c$ on $L^q(M,E_0^h)\cap d_c^{-1}(L^1(M,E_0^{h-1}))$. This applies in particular to $M=\he n$.

Fix $h=0,\ldots,2n$. Let $C$ denote the space of left-invariant Rumin $2n+1-h$-forms on $\he n$. Integrating closed $L^1$-forms $\omega$ against left-invariant forms $\beta$ defines a bilinear map
$$
(\omega,\beta)\mapsto \int_{\he n}\omega\wedge\beta,\quad (L^1(\he n,H_0^h)\cap \mathrm{ker}(d_c))\times C\to\R,
$$ 
whence a map
$$
I:L^1(\he n,H_0^h)\cap \mathrm{ker}(d_c)\to C^*.
$$
Pick $d_c$-closed $L^1$ forms $(\psi_1,\ldots,\psi_k)$ such that $(I(\psi_1),\ldots,I(\psi_k))$ is a basis of its image. 

Let $\omega$ be a $d_c$-closed $h$-form. There exist real numbers $\lambda_1,\ldots,\lambda_k$ such that
$$
I(\omega)=\sum_{i=1}^{k}\lambda_i I(\psi_i).
$$
Then $\omega_0=\omega-\sum_{i=1}^{k}\lambda_i \psi_i$ is $d_c$-closed and belongs to $L_0^1$. According to Theorem \ref{core}, there exists an $h-1$-form $\phi\in L^q(\he n,E_0^{h-1})$ such that $\omega_0=d_c\phi$ (here, $q=Q/Q-1$ or $Q/Q-2$ depending on $h$). This shows that the dimension of $\ell^{q,1}H^h(\he n)$ is at most $k$.
\end{proof}

\section*{Acknowledgments}
B.~F. and A.~B. are supported by the University of Bologna, funds for selected research topics, and by MAnET Marie Curie
Initial Training Network, by 
GNAMPA of INdAM (Istituto Nazionale di Alta Matematica ``F. Severi''), Italy, and by PRIN of the MIUR, Italy.

P.P. is supported by MAnET Marie Curie
Initial Training Network, by Agence Nationale de la Recherche, ANR-10-BLAN 116-01 GGAA and ANR-15-CE40-0018 SRGI. P.P. gratefully acknowledges the hospitality of Isaac Newton Institute, of EPSRC under grant EP/K032208/1, and of Simons Foundation.

\bibliographystyle{amsplain}

\bibliography{BFP3_final}

\providecommand{\bysame}{\leavevmode\hbox to3em{\hrulefill}\thinspace}
\providecommand{\MR}{\relax\ifhmode\unskip\space\fi MR }
\providecommand{\MRhref}[2]{%
  \href{http://www.ams.org/mathscinet-getitem?mr=#1}{#2}
}
\providecommand{\href}[2]{#2}
\begin{thebibliography}{10}

\bibitem{BBF}
Annalisa Baldi, Marilena Barnabei, and Bruno Franchi, \emph{A recursive basis
  for primitive forms in symplectic spaces and applications to {H}eisenberg
  groups}, Acta Math. Sin. (Engl. Ser.) \textbf{32} (2016), no.~3, 265--285.
  \MR{3456421}

\bibitem{BF7}
Annalisa Baldi and Bruno Franchi, \emph{Sharp a priori estimates for div-curl
  systems in {H}eisenberg groups}, J. Funct. Anal. \textbf{265} (2013), no.~10,
  2388--2419. \MR{3091819}

\bibitem{BFP}
Annalisa Baldi, Bruno Franchi, and Pierre Pansu, \emph{Gagliardo-{N}irenberg
  inequalities for differential forms in {H}eisenberg groups}, Math. Ann.
  \textbf{365} (2016), no.~3-4, 1633--1667. \MR{3521101}

\bibitem{BFP4}
\bysame, \emph{$l^1$-{P}oincar\'e and {S}obolev inequalities for differential
  forms in {E}uclidean spaces}, Preprint, 2018.

\bibitem{BFP2}
\bysame, \emph{{P}oincar\'e and {S}obolev inequalities for differential forms
  in {H}eisenberg groups}, Preprint, 2018.

\bibitem{BFTT}
Annalisa Baldi, Bruno Franchi, Nicoletta Tchou, and Maria~Carla Tesi,
  \emph{Compensated compactness for differential forms in {C}arnot groups and
  applications}, Adv. Math. \textbf{223} (2010), no.~5, 1555--1607.

\bibitem{BFT3}
Annalisa Baldi, Bruno Franchi, and Maria~Carla Tesi, \emph{Hypoellipticity,
  fundamental solution and {L}iouville type theorem for matrix--valued
  differential operators in {C}arnot groups}, J. Eur. Math. Soc. \textbf{11}
  (2009), no.~4, 777--798.

\bibitem{BBD}
Philippe Benilan, Haim Brezis, and Michael~G. Crandall, \emph{A semilinear
  equation in {$L^{1}(R^{N})$}}, Ann. Scuola Norm. Sup. Pisa Cl. Sci. (4)
  \textbf{2} (1975), no.~4, 523--555. \MR{0390473}

\bibitem{berg_lofstrom}
J\"{o}ran Bergh and J\"{o}rgen L\"{o}fstr\"{o}m, \emph{Interpolation spaces.
  {A}n introduction}, Springer-Verlag, Berlin-New York, 1976, Grundlehren der
  Mathematischen Wissenschaften, No. 223. \MR{0482275}

\bibitem{Bourgain-Brezis-JEMS}
Jean Bourgain and Ha{\"{\i}}m Brezis, \emph{New estimates for elliptic
  equations and {H}odge type systems}, J. Eur. Math. Soc. (JEMS) \textbf{9}
  (2007), no.~2, 277--315. \MR{2293957 (2009h:35062)}

\bibitem{CG}
Luca Capogna and Nicola Garofalo, \emph{Boundary behavior of nonnegative
  solutions of subelliptic equations in {NTA} domains for
  {C}arnot-{C}arath\'{e}odory metrics}, J. Fourier Anal. Appl. \textbf{4}
  (1998), no.~4-5, 403--432. \MR{1658616}

\bibitem{CGN_novosibirsk}
Luca Capogna, Nicola Garofalo, and Duy-Minh Nhieu, \emph{Examples of uniform
  and {NTA} domains in {C}arnot groups}, Proceedings on {A}nalysis and
  {G}eometry ({R}ussian) ({N}ovosibirsk {A}kademgorodok, 1999), Izdat. Ross.
  Akad. Nauk Sib. Otd. Inst. Mat., Novosibirsk, 2000, pp.~103--121.
  \MR{1847513}

\bibitem{dR}
Georges de~Rham, \emph{Vari\'{e}t\'{e}s diff\'{e}rentiables. {F}ormes,
  courants, formes harmoniques}, Actualit\'{e}s Sci. Ind., no. 1222 = Publ.
  Inst. Math. Univ. Nancago III, Hermann et Cie, Paris, 1955. \MR{0068889}

\bibitem{evans_gariepy}
Lawrence~C. Evans and Ronald~F. Gariepy, \emph{Measure theory and fine
  properties of functions}, revised ed., Textbooks in Mathematics, CRC Press,
  Boca Raton, FL, 2015. \MR{3409135}

\bibitem{federer_fleming}
Herbert Federer and Wendell~H. Fleming, \emph{Normal and integral currents},
  Ann. of Math. (2) \textbf{72} (1960), 458--520. \MR{0123260}

\bibitem{folland}
Gerald~B. Folland, \emph{Subelliptic estimates and function spaces on nilpotent
  {L}ie groups}, Ark. Mat. \textbf{13} (1975), no.~2, 161--207. \MR{MR0494315
  (58 \#13215)}

\bibitem{folland_stein}
Gerald~B. Folland and Elias~M. Stein, \emph{Hardy spaces on homogeneous
  groups}, Mathematical Notes, vol.~28, Princeton University Press, Princeton,
  N.J., 1982. \MR{MR657581 (84h:43027)}

\bibitem{FGaW}
Bruno Franchi, Sylvain Gallot, and Richard~L. Wheeden, \emph{Sobolev and
  isoperimetric inequalities for degenerate metrics}, Math. Ann. \textbf{300}
  (1994), no.~4, 557--571. \MR{1314734 (96a:46066)}

\bibitem{FLW_grenoble}
Bruno Franchi, Guozhen Lu, and Richard~L. Wheeden, \emph{Representation
  formulas and weighted {P}oincar\'e inequalities for {H}\"ormander vector
  fields}, Ann. Inst. Fourier (Grenoble) \textbf{45} (1995), no.~2, 577--604.
  \MR{1343563 (96i:46037)}

\bibitem{FPS}
Bruno Franchi, Valentina Penso, and Raul Serapioni, \emph{Remarks on
  {L}ipschitz domains in {C}arnot groups}, Geometric control theory and
  sub-{R}iemannian geometry, Springer INdAM Ser., vol.~5, Springer, Cham, 2014,
  pp.~153--166. \MR{3205101}

\bibitem{FSSC_houston}
Bruno Franchi, Raul Serapioni, and Francesco Serra~Cassano,
  \emph{Meyers-{S}errin type theorems and relaxation of variational integrals
  depending on vector fields}, Houston J. Math. \textbf{22} (1996), no.~4,
  859--890. \MR{1437714}

\bibitem{GN}
Nicola Garofalo and Duy-Minh Nhieu, \emph{Isoperimetric and {S}obolev
  inequalities for {C}arnot-{C}arath\'eodory spaces and the existence of
  minimal surfaces}, Comm. Pure Appl. Math. \textbf{49} (1996), no.~10,
  1081--1144. \MR{1404326 (97i:58032)}

\bibitem{GromovCC}
Mikhael Gromov, \emph{Carnot-{C}arath\'eodory spaces seen from within},
  Sub-Riemannian geometry, Progr. Math., vol. 144, Birkh\"auser, Basel, 1996,
  pp.~79--323. \MR{MR1421823 (2000f:53034)}

\bibitem{HN}
Bernard Helffer and Jean Nourrigat, \emph{Hypoellipticit\'e maximale pour des
  op\'erateurs polyn\^omes de champs de vecteurs}, Progress in Mathematics,
  vol.~58, Birkh\"auser Boston Inc., Boston, MA, 1985. \MR{MR897103
  (88i:35029)}

\bibitem{IL}
Tadeusz Iwaniec and Adam Lutoborski, \emph{Integral estimates for null
  {L}agrangians}, Arch. Rational Mech. Anal. \textbf{125} (1993), no.~1,
  25--79. \MR{MR1241286 (95c:58054)}

\bibitem{LS}
Loredana Lanzani and Elias~M. Stein, \emph{A note on div curl inequalities},
  Math. Res. Lett. \textbf{12} (2005), no.~1, 57--61. \MR{2122730
  (2005m:58001)}

\bibitem{LN}
Rui~Lin Long and Fu~Sheng Nie, \emph{Weighted {S}obolev inequality and
  eigenvalue estimates of {S}chr\"{o}dinger operators}, Harmonic analysis
  ({T}ianjin, 1988), Lecture Notes in Math., vol. 1494, Springer, Berlin, 1991,
  pp.~131--141. \MR{1187073}

\bibitem{lu_acta_sinica}
Guozhen Lu, \emph{Polynomials, higher order {S}obolev extension theorems and
  interpolation inequalities on weighted {F}olland-{S}tein spaces on stratified
  groups}, Acta Math. Sin. (Engl. Ser.) \textbf{16} (2000), no.~3, 405--444.
  \MR{1787096}

\bibitem{mattila}
Pertti Mattila, \emph{Geometry of sets and measures in {E}uclidean spaces},
  Cambridge Studies in Advanced Mathematics, vol.~44, Cambridge University
  Press, Cambridge, 1995, Fractals and rectifiability. \MR{1333890}

\bibitem{mcduff_salamon}
Dusa McDuff and Dietmar Salamon, \emph{Introduction to symplectic topology},
  second ed., Oxford Mathematical Monographs, The Clarendon Press, Oxford
  University Press, New York, 1998. \MR{1698616}

\bibitem{mitrea_mitrea_monniaux}
Dorina Mitrea, Marius Mitrea, and Sylvie Monniaux, \emph{The {P}oisson problem
  for the exterior derivative operator with {D}irichlet boundary condition in
  nonsmooth domains}, Commun. Pure Appl. Anal. \textbf{7} (2008), no.~6,
  1295--1333. \MR{2425010}

\bibitem{montgomery}
Richard Montgomery, \emph{A tour of subriemannian geometries, their geodesics
  and applications}, Mathematical Surveys and Monographs, vol.~91, American
  Mathematical Society, Providence, RI, 2002. \MR{MR1867362 (2002m:53045)}

\bibitem{monti_thesis}
Roberto Monti, \emph{{Distances, boundaries and surface measures in
  {C}arnot-{C}arath\'eodory spaces}}, Ph.D. thesis, School of Sciences,
  University of Trento, Italy, 2001.

\bibitem{MoMo}
Roberto Monti and Daniele Morbidelli, \emph{Regular domains in homogeneous
  groups}, Trans. Amer. Math. Soc. \textbf{357} (2005), no.~8, 2975--3011.
  \MR{2135732}

\bibitem{nhieu_AMPA}
Duy-Minh Nhieu, \emph{The {N}eumann problem for sub-{L}aplacians on {C}arnot
  groups and the extension theorem for {S}obolev spaces}, Ann. Mat. Pura Appl.
  (4) \textbf{180} (2001), no.~1, 1--25. \MR{1848049}

\bibitem{Ornstein}
Donald Ornstein, \emph{A non-equality for differential operators in the
  {$L_{1}$} norm}, Arch. Rational Mech. Anal. \textbf{11} (1962), 40--49.
  \MR{0149331}

\bibitem{PPcup}
Pierre Pansu, \emph{Cup-products in $l^{q,p}$-cohomology: discretization and
  quasi-isometry invariance}, Preprint, 2017.

\bibitem{rumin_jdg}
Michel Rumin, \emph{Formes diff\'erentielles sur les vari\'et\'es de contact},
  J. Differential Geom. \textbf{39} (1994), no.~2, 281--330. \MR{MR1267892
  (95g:58221)}

\bibitem{rumin_palermo}
\bysame, \emph{An introduction to spectral and differential geometry in
  {C}arnot-{C}arath\'eodory spaces}, Rend. Circ. Mat. Palermo (2) Suppl.
  \textbf{75} (2005), 139--196. \MR{MR2152359 (2006g:58053)}

\bibitem{saka}
Koichi Saka, \emph{{B}esov spaces and {S}obolev spaces on a nilpotent {L}ie
  group}, T\^{o}hoku Math. J. (2) \textbf{31} (1979), no.~4, 383--437.
  \MR{558675}

\bibitem{schwarz}
G{\"u}nter Schwarz, \emph{Hodge decomposition---a method for solving boundary
  value problems}, Lecture Notes in Mathematics, vol. 1607, Springer-Verlag,
  Berlin, 1995. \MR{MR1367287 (96k:58222)}

\bibitem{Stein}
Elias~M. Stein, \emph{Harmonic analysis: real-variable methods, orthogonality,
  and oscillatory integrals}, Princeton Mathematical Series, vol.~43, Princeton
  University Press, Princeton, NJ, 1993, With the assistance of Timothy S.
  Murphy, Monographs in Harmonic Analysis, III. \MR{MR1232192 (95c:42002)}

\bibitem{Tripaldi}
Francesca Tripaldi, \emph{Averages and the $\ell^{q,1}$-cohomology of
  {H}eisenberg groups}, in preparation, 2018.

\bibitem{vS2014}
Jean Van~Schaftingen, \emph{Limiting {B}ourgain-{B}rezis estimates for systems
  of linear differential equations: theme and variations}, J. Fixed Point
  Theory Appl. \textbf{15} (2014), no.~2, 273--297. \MR{3298002}

\bibitem{VarSalCou}
Nicholas~Th. Varopoulos, Laurent Saloff-Coste, and Thierry Coulhon,
  \emph{Analysis and geometry on groups}, Cambridge Tracts in Mathematics, vol.
  100, Cambridge University Press, Cambridge, 1992. \MR{MR1218884 (95f:43008)}

\bibitem{VG}
S.~K. Vodop'yanov and A.~V. Greshnov, \emph{On the continuation of functions of
  bounded mean oscillation on spaces of homogeneous type with intrinsic
  metric}, Sibirsk. Mat. Zh. \textbf{36} (1995), no.~5, 1015--1048, i.
  \MR{1373594}

\end{thebibliography}

\bigskip
\tiny{
\noindent
Annalisa Baldi and Bruno Franchi 
\par\noindent
Universit\`a di Bologna, Dipartimento
di Matematica\par\noindent Piazza di
Porta S.~Donato 5, 40126 Bologna, Italy.
\par\noindent
e-mail:
annalisa.baldi2@unibo.it, 
bruno.franchi@unibo.it.
}

\medskip

\tiny{
\noindent
Pierre Pansu 
\par\noindent Laboratoire de Math\'ematiques d'Orsay,
\par\noindent Universit\'e Paris-Sud, CNRS,
\par\noindent Universit\'e
Paris-Saclay, 91405 Orsay, France.
\par\noindent
e-mail: pierre.pansu@math.u-psud.fr
}

\end{document}